\documentclass[11pt]{article}
\usepackage{amssymb,amsmath}
\usepackage{mathrsfs}
\usepackage{latexsym}
\usepackage{amssymb}
\usepackage{amsmath}
\usepackage{amsthm}
\usepackage{colortbl}
\allowdisplaybreaks

\catcode`!=11
\let\!int\int \def\int{\displaystyle\!int}
\let\!lim\lim \def\lim{\displaystyle\!lim}
\let\!sum\sum \def\sum{\displaystyle\!sum}
\let\!sup\sup \def\sup{\displaystyle\!sup}
\let\!inf\inf \def\inf{\displaystyle\!inf}
\let\!cap\cap \def\cap{\displaystyle\!cap}
\let\!max\max \def\max{\displaystyle\!max}
\let\!min\min \def\min{\displaystyle\!min}
\catcode`!=12

\newcommand{\be}{\begin{equation}}
\newcommand{\ee}{\end{equation}}
\newcommand{\ba}{\begin{eqnarray}}
\newcommand{\ea}{\end{eqnarray}}
\newtheorem{theorem}{\bf Theorem}[section]
\newtheorem{lemma}{\bf Lemma}[section]

\newtheorem{corollary}{\bf Corollary}[section]
\newtheorem{proposition}{\bf Proposition}[section]
\newtheorem{remark}{\bf Remark}[section]

\catcode`@=11
\@addtoreset{equation}{section}
\catcode`@=12

\newcommand{\R}{\mathbb R}

\newcommand{\N}{\mathbb N}

\def\u{\hat u_j}
\def\g{\hat g_j}
\def\h{\hat h_j}
\def\lj{\lambda _j}

\begin{document}
\title{Exact controllability of the linear Zakharov-Kuznetsov equation}
\author{Mo Chen\footnote{This author  is supported by NSFC Grant (11701078) and China Scholarship Council.}
\\[2mm]
\small School of Mathematics and Statistics,
\\
\small Center for Mathematics and Interdisciplinary Sciences,
\\
\small Northeast Normal University, Changchun, 130024, P. R. China
\\
\small chenmochenmo.good@163.com
\\[2mm]
Lionel Rosier\footnote{This author is supported  by the ANR project Finite4SoS (ANR-15-CE23-0007)}
\\
\small $^a$ Universit\'e du Littoral C\^ote d'Opale
\\
\small Laboratoire de Math\'ematiques Pures et Appliqu\'ees  J. Liouville, 
\\
\small BP 699, F-62228 Calais, France. 
\\
\small $^b$ CNRS FR 2956, France
\\
\small Lionel.Rosier@univ-littoral.fr}

\date{}
\maketitle

\vbox to -13truemm{}
\vspace{1.cm}
\begin{abstract}
We consider the linear Zakharov-Kuznetsov equation on a rectangle with a left Dirichlet boundary control.
Using the flatness approach, we prove the null controllability of this equation and provide a space of analytic reachable states. \\


\end{abstract}

\section{Introduction}

The Zakharov-Kuznetsov (ZK) equation 
\be
\label{I1}
u_t +au_x +\Delta u_x +uu_x=0, 
\ee
provides a model for the propagation of nonlinear ionic-sonic waves in a plasma. 
In \eqref{I1},   $x,t\in \R$ and $y\in \R ^d$ (with $d\in \{1, 2\}$) are the independent variables, $u=u(x,y,t)$ is the unknown,  
$u_t=\partial u/\partial t$, $u_x=\partial u/\partial x$, $\Delta u=\partial ^2 u /\partial x^2
+\sum_{i=1}^d \partial ^2 u /\partial y_i^2$, and the constant $a>0$ stands for the sound velocity. 
The ZK equation is, from the mathematical point of view, a natural extension to $\R ^{d+1}$ of the famous Korteweg-de Vries equation
\be
\label{I2}
z_t + az_x + z_{xxx} + zz_x =0, 
\ee 
which has been extensively studied from the control point of view
(see e.g. the surveys \cite{cerpa,RZsurvey}).
If we focus on the situation where 
\eqref{I2} is supplemented with the following boundary conditions
\be
\label{I3}
z(0,t)=h(t), \quad z(L,t)=z_x(L,t)=0, 
\ee
where $L>0$ is a given number and $h$ is the control input, then it was proved in 
\cite{GG,R2004}  that \eqref{I2}-\eqref{I3} was null controllable on the domain $(0,L)$.
Due to the smoothing effect, with such a control at the left endpoint the exact controllability can only hold in a space of analytic functions. 

More recently, a space of analytic reachable states was provided  in \cite{MRRR}  for the linearized
KdV equation
\[
z_t+z_x+z_{xxx}=0
\]
with the  same boundary conditions as in \eqref{I3}. The method of proof was based on the flatness approach, as
introduced in \cite{MRR} to study the reachable states of the heat equation. The aim of the paper is to extend the results given in \cite{MRRR} to the  ZK  equation. 

The wellposedness of various initial boundary value problems for ZK were studied in 
\cite{F1,F2,LS,LPS,ST,STW}. Some unique continuation property for ZK derived with a Carleman estimate was done in \cite{C}. 
Exact controllability results for ZK in the same spirit as those for KdV in \cite{R1997} are given in 
\cite{F2,PRST}.

Here, we limit ourselves to the case $d=1$, so that $y\in \R$. 
By a translation, we can assume without loss of generality that $x\in (-1,0)$ (this will be more convenient when using series to represent the solutions).
We set $\Omega:=(-1,0)\times (0,1)$. 
 The paper is concerned with the control properties  of the system: 
\ba
u_t + u_{xxx} +u_{xyy} +au_x =0,&& (x,y)\in \Omega , \   t\in (0,T), \label{A1}\\
u(0,y,t)=u_x(0,y,t)=0,&&  y\in (0,1),\  t\in (0,T), \label{A2}\\
u(-1,y, t)=h(y,t),&& y\in (0,1),\ t\in (0,T), \label{A3}\\
u(x,0,t)=u(x,1,t)=0,&& x\in (-1,0), \ t\in (0,T),\label{A3bis}\\
u(x,y,0)=u_0(x,y),&& (x,y)\in \Omega  ,\label{A4}
\ea
where $u_0=u_0(x,y) $ is the initial data and $h=h(y,t)$ is the control input. \\

We shall address the following issues: \\[3mm]
1. (Null controllability) Given any $u_0\in L^2(\Omega )$, can we find a control $h$ such that the solution $u$ of \eqref{A1}-\eqref{A4} 
satisfies $u(.,T)=0$?\\
2. (Reachable states) Given any $u_1\in {\mathcal R}$ (a subspace of $L^2(\Omega )$ defined thereafter), can we find a control $h$ such that the solution $u$ of \eqref{A1}-\eqref{A4} with $u_0=0$ satisfies $u(.,T)=u_1$?\\    
We shall investigate both issues by the flatness approach and derive an exact controllability in $\mathcal R$ by combining
our results.

To state our result, we need introduce notations. A function $u\in C^\infty ( [ t_1,t_2])$ is said to be {\em Gevrey of 
order $s\ge 0$ on $[t_1,t_2]$}  if there exist some constant $C,R\ge 0$ such that 
\[
\vert \partial _t ^n u(t) \vert \le C \frac{ (n!) ^s}{R^n}\quad \forall n\in \N, \ \forall t\in [t_1,t_2]. 
\] 
The set of functions Gevrey of order $s$ on $[t_1,t_2]$ is denoted by $G^s ([t_1,t_2])$.  A function 
$u\in C^\infty ([x_1,x_2] \times [y_1,y_2]\times [t_1,t_2] )$ is said to be Gevrey of order $s_1$ in $x$, $s_2$ in $y$ and $s_3$ in $t$ on $[x_1,x_2]\times [y_1,y_2]\times [t_1,t_2]$
if there exist some constants $C,R_1,R_2,R_3>0$ such that 
\[
\vert \partial _x^{n_1}\partial _y ^{n_2}
 \partial _t ^{n_3}  u(x,y,t) \vert \le C \frac{ (n_1 ! )^{s_1}  (n_2 !)^{s_2}   (n_3 !)^{s_3} }{R_1^{n_1} R_2^{n_2}  R_3^{n_3}  }
\quad \forall n_1,n_2,n_3\in \N, \ \forall (x,y,t)\in [x_1,x_2] \times [y_1,y_2] \times [t_1,t_2]. 
\] 
The set of functions Gevrey of order $s_1$ in $x$, $s_2$ in $y$ and $s_3$ in $t$  on $[x_1,x_2] \times [y_1,y_2]
\times [t_1,t_2]$ is denoted by 
$G^{s_1,s_2,s_3}  ( [x_1,x_2]\times  [y_1,y_2]\times [t_1,t_2])$.  

The first main result in this paper is a null controllability result with a control input in a Gevrey class.

\begin{theorem}\label{T1}
Let $u_{0}\in L^{2}(\Omega)$ and $s\in[\frac{3}{2},2)$. Then there exists a control input $h\in G^{\frac{s}{2},s}([0,1]\times[0,T])$ such that the solution $u$ of \eqref{A1}-\eqref{A4}  satisfies $u(\cdot,\cdot,T)$=0. Furthermore, it holds that
\begin{equation*}
u\in C([0,T];L^{2}(\Omega))\cap G^{\frac{s}{2},\frac{s}{2},s}([-1,0]\times[0,1]\times[\varepsilon,T]),~~~\forall~\varepsilon\in(0,T).
\end{equation*}
\end{theorem}
Introduce the differential operator
$$Pu:=\triangle u_{x}+au_{x}$$
and the following space 
\begin{eqnarray*}
\mathcal{R}_{R_{1},R_{2}} :=\{u\in C^\infty ([-1,0]\times[0,1]); \ \exists C>0, \ \vert \partial_{x}^{p}\partial_{y}^{q} \, u (x,y)|\leq C\frac{(p!)^{\frac{2}{3}}(q!)^{\frac{2}{3}}}{R_{1}^{p}R_{2}^{q}}\quad 
\forall p,q\in\mathbb{N},\ \forall (x,y)\in \overline{\Omega} ,  \\
\quad \textrm{ and } P^{n}u(0,y)=\partial_{x}P^{n}u(0,y)=P^{n}u(x,0)=P^{n}u(x,1)=0,\ \ \forall n\in 
{\mathbb N} , \ 
\forall x\in [-1,0],\ \forall y\in [0,1] \} .
\end{eqnarray*}

Our second main result provides a set of reachable states for system \eqref{A1}-\eqref{A4}. 
\begin{theorem}\label{T2}
Let $R_{0} :=\sqrt[3]{9(a+2)}e^{(3e)^{-1}}$, and let $R_1,R_2\in (R_0,+\infty )$. Then for any 
$u_{1}\in \mathcal{R}_{R_{1},R_{2}}$, there exists a control input $h\in G^{1,2}([0,1]\times[0,T])$ such that the solution $u$ of \eqref{A1}-\eqref{A4}  with $u_{0}=0$ satisfies $u(\cdot,\cdot,T)=u_{1}$. Furthermore, $u\in G^{1,1,2}([-1,0]\times[0,1]\times[0,T])$, and the trajectory $u=u(x,y,t)$  and the control $h=h(y,t)$ can be expanded as series: 
\begin{eqnarray} \label{17}
u(x,y,t)&=&\sum\limits_{j=1}\limits^{\infty}\sum\limits_{i=0}\limits^{\infty}g_{i,j}(x)z_{j}^{(i)}(t)e_{j}(y),\\
h(y,t)&=&\sum\limits_{j=1}\limits^{\infty}\sum\limits_{i=0}\limits^{\infty}g_{i,j}(-1)z_{j}^{(i)}(t)e_{j}(y). 
\label{17bis}
\end{eqnarray}
\end{theorem}
We refer the reader to Section 2 for the definitions of the functions $g_{i,j}$ ($i\ge 0$, $j\ge 1$) and of the functions $e_j$
 ($j\ge 1$). 

Combining Theorem \ref{T1} and Theorem \ref{T2}, we obtain the 
following result which implies the exact controllability of \eqref{A1}-\eqref{A4} 
in ${\mathcal R}_{R_1,R_2}$  for $R_1>R_0$ and  $R_2>R_0$.

\begin{corollary}
\label{cor1}
Let $R_{0} :=\sqrt[3]{9(a+2)}e^{(3e)^{-1}}$, and let $R_1,R_2\in (R_0,+\infty )$. 
Let $u_0\in L^2(\Omega )$ and $u_1\in {\mathcal R}_{R_1,R_2}$. 
Then there exists $h\in G^{1,2}([0,1]\times[0,T])$  such that the solution of \eqref{A1}-\eqref{A4} satisfies  $u(.,T)=u_1$. 
\end{corollary}   

The paper is outlined as follows. Section 2 introduces the eigenfunctions $e_j$, the generating functions $g_{i,j}$,
and provide some estimates needed in the sequel. The null controllability of ZK is established in Section 3, while the 
reachable states of ZK are investigated in Section 4.

\section{Preliminaries}
First we introduce the operator
$$Au:=-Pu=-\triangle u_{x}-au_{x}$$
with domain
\begin{equation*}
\begin{split}
\mathcal{D}(A)=\{&u\in L^{2}(\Omega);\ Pu\in L^{2}(\Omega),~u(-1,y)=u(0,y)=u_{x}(0,y)=0~\textrm{for a.e.}~y\in(0,1)~\textrm{and}\\
&u(x,0)=u(x,1)=0~\textrm{for a.e.}~x\in(-1,0)\}.
\end{split}
\end{equation*}
It is well-known (see e.g. \cite{STW}) that the operator $A$ generates a semigroup of contractions in $L^2(\Omega )$.   
In what follows, we denote  $\Vert f\Vert _{\mathcal{D}(A)} = \Vert f\Vert _{L^2(\Omega )}  + \Vert A \, f\Vert _{L^2(\Omega )} $ for all 
$f\in {\mathcal{D}} (A)$.

It would be natural to expect, as for KdV, that the domain $\mathcal{D} (A)$ coincide with the set 
\[
\{ u\in H^3 (\Omega ) \cap H^1_0 (\Omega); \ \  u_x (0,y)=0 \textrm{ for a.e. } y\in (0,1) \}, 
\]
but this is not the case. The best description (up to date) of $\mathcal{D} (A)$ is given in the following lemma. 
\begin{lemma}
We have the following inclusions:
\ba
&&\{ u\in H^2(\Omega )\cap H^1_0(\Omega ); \  u_x\in H^2(\Omega ) \textrm{ and } u_x(0, y)=0\textrm { for a.e. }  y\in (0,1)\} \subset 
\mathcal{D}(A), \qquad \label{AA1} \\
&&\mathcal{D} (A) \subset \{ u\in H^2 (\Omega )\cap H^1_0(\Omega ); \ (x+1)u_x\in H^2(\Omega )\cap H^1_0(\Omega )\}.  \qquad
\label{AA2}
\ea
\end{lemma} 
\begin{proof}
The inclusion \eqref{AA1} is obvious. For \eqref{AA2}, it follows from \cite[Proposition 2]{STW}  that 
$\mathcal{D} (A) \subset H^2(\Omega ) \cap H^1_0(\Omega )$. If $u\in \mathcal{D} (A)$, then $f:=\Delta u_x + au_x\in L^2(\Omega )$ 
and hence
\[
\Delta ((x+1)u_x)=(x+1)\Delta u_x + 2u_{xx}=(x+1)(f-au_x)+2u_{xx}\in L^2(\Omega ). 
\]  
On the other hand, we claim that $(x+1)u_x\in H^1_0(\Omega )$. Indeed, $u_x\in H^1(\Omega)$ and hence $(x+1)u_x\in H^1(\Omega )$. Moreover, 
$u(.,0)=u(.,1)=0$ in $H^\frac{3}{2}(-1,0)$ gives $u_x(.,0)=u_x(.,1)=0$ in $H^\frac{1}{2}(-1,0)$, and finally $((x+1)u_x)(-1,.)= u_x(0,.)=0$ in 
$H^\frac{1}{2}(0,1)$. By the classical boundary $H^2$ regularity result for the Dirichlet problem on a Lipschitz domain, we infer that 
$(x+1)u_x\in H^2(\Omega ) \cap H^1_0(\Omega )$. 
\end{proof}
\begin{remark}
It can be shown that the inclusion \eqref{AA1} is strict. 
\end{remark}

The following lemmas will be used several times thereafter.
\begin{lemma}
\label{lemma0}\
For any $n\in \N ^*$ and any $f\in {\mathcal D} (A^n)$ 
with $A^i  f\in H^{2(n-i)} (\Omega)$ for $i=0,1,...,n$,  
we have
\begin{equation}\label{22a}
\partial_{y}^{2p}f(x,0)=\partial_{y}^{2p}f(x,1)=0,~~~\forall x\in [-1,0],\  \forall p\in \{  0 ,..., n-1 \} .
\end{equation}
\end{lemma}
\begin{proof}
We proceed by induction on $n$. For $n=1$, the property \eqref{22a} is obvious since $f\in {\mathcal D}(A)$. 
Assume now that \eqref{22a} is true for $n-1\geq 1$. If $f\in {\mathcal D}(A^n)$ with $A^i  f\in H^{2(n-i)} (\Omega)$ for $i=0,1,...,n$, 
then $Pf=-A f \in {\mathcal D}(A^{n-1})$ with $A^i  Pf =-A^{i+1} f \in H^{2(n-i-1)} (\Omega)$ for $i=0,1,...,n-1$, , so that
by \eqref{22a} applied to $Pf$ and $p=n-2$
\begin{equation*}
\partial_{y}^{2n-4}Pf(x,0)=\partial_{y}^{2n-4}Pf(x,1)=0.
\end{equation*}
This implies
\begin{eqnarray}
&&\partial_{x}^{3}\partial_{y}^{2n-4}f(x,0)+\partial_{x}\partial_{y}^{2n-2}f(x,0) +a  \partial _x \partial _y ^{2n-4} f(x,0)=0,
\label{PPP1}\\
&&\partial_{x}^{3}\partial_{y}^{2n-4}f(x,1)+\partial_{x}\partial_{y}^{2n-2}f(x,1) +a \partial _x \partial _y ^{2n-4} f(x,1)
=0.\label{PPP2}
\end{eqnarray}
Since \eqref{22a} is true for $n-1$, we obtain that $\partial_{y}^{2p}f(x,0) =\partial_{y}^{2p}f(x,1)=0$
for $p=0,1,...,n-2$, and hence (taking $p=n-2$ and using \eqref{PPP1}-\eqref{PPP2})
\begin{equation*}
\partial_{x}\partial_{y}^{2n-2}f(x,0)=\partial_{x}\partial_{y}^{2n-2}f(x,1)=0.
\end{equation*}
This means that we have for some constants $C_{1}$ and $C_{2}$
\[
\partial_{y}^{2n-2}f(x,0)=C_{1},~\partial_{y}^{2n-2}f(x,1)=C_{2}\quad \forall x\in [-1,0].
\]
Note that $\partial _y^{2n-2} f\in H^2(\Omega )\subset C(\overline{\Omega})$.
On the other hand, it follows from the assumption $f\in {\mathcal D }(A) $ that
\begin{equation*}
\partial_{y}^{2n-2}f(0,y)=0\quad \forall y\in [0,1].
\end{equation*}
\noindent
Taking $y=0$ and next $y=1$, we see that $C_{1}=C_{2}=0$. The proof of Lemma \ref{lemma0}  is complete.
\end{proof}
\begin{remark} 
It will be proved in Proposition \ref{P1} (see below) that ${\mathcal D}(A^n)\subset H^{2n}(\Omega) $ for all $n\in \N$, so 
that the conclusion of Lemma  \ref{lemma0} will be still valid when assuming solely that $f\in {\mathcal D}(A^n)$. 
\end{remark}

The following lemma is classical. Its proof is omitted.
\begin{lemma}
\label{lemma0bis} 
Let $A'=\partial _y^2$ with domain ${\mathcal D}(A')=H^2(0,1)\cap H^1_0(0,1)$. Then for any $m\in {\mathbb N} ^*$, it holds
\[
{\mathcal D}({A'}^\frac{m}{2}) =\{ g\in H^{m}(0,1); \ g^{(2p)}(0)=g^{(2p)}(1)=0\textrm{ for } 0\le p\le \frac{m-1}{2}\}  \cdot
\]
Let $h\in L^2(0,1)$ be decomposed as $h(y)=\sum_{j=1}^\infty c_je_j(y)$, and let $m\in {\mathbb N}^*$. Then 
\[
h\in {\mathcal D} ({A'}^\frac{m}{2}) \iff \sum_{j=1}^\infty \vert \lambda _j ^\frac{m}{2} c_j \vert^2 <\infty. 
\]
Furthermore, for any $h\in {\mathcal D} ({A'}^\frac{m}{2})$, we have
\[
\Vert h^{(q)} \Vert ^2_{L^2(0,1)} =   \sum_{j=1}^\infty \lambda _j ^q\vert  c_j \vert^2 \quad \forall q\in \{ 0, ..., m\}.  
\]
\end{lemma}

We are in a position to state the main result in this section. 
\begin{proposition}\label{P1} For any $n\in\N$, it holds ${\mathcal D}(A^n) \subset H^{2n} (\Omega )$. Furthermore, 
there exists a constant $B\ge 1$ such that
\begin{equation}\label{8}
\| u\|_{H^{2n} (\Omega ) } \leq B^{n}\sum\limits_{i=0}\limits^{n}\|P^{i}u \| _{L^2(\Omega )} ,~~~\forall n\in\N , \  \forall u \in \mathcal{D}(A^{n}).
\end{equation}
\end{proposition}
\begin{proof}
Let $\{e_{j}\}_{j\geq1}$ be an orthonormal basis in $L^{2}(0,1)$ such that $e_{j}$ is an eigenfunction for the Dirichlet Laplacian on $(0,1)$, $\lambda_{j}$ being the corresponding eigenvalue; that is 
\begin{equation*}
\begin{split}
&-e_{j}^{\prime\prime}(y)=\lambda_{k}e_{j}(y),\\
&e_{j}(0)=e_{j}(1)=0.
\end{split}
\end{equation*}
A classical choice is $e_{j}(y)=\sqrt{2}\sin(j\pi y)$ and $\lambda_{j}=( j\pi )^{2}$ for $j\ge 1$.
Following \cite{STW}, we decompose any function $u\in L^2(\Omega )$ as
\[
u(x,y)=\sum_{j=1}^\infty \u (x) e_j(y).
\]
Note that $\Vert u\Vert ^2_{L^2(\Omega)} = \sum_{j=1}^\infty \Vert \u\Vert ^2$, where we denote
$\Vert h\Vert = \Vert h\Vert _{L^2( -1,0)}$ for all $h\in L^2(-1,0)$ for the sake of simplicity.
If $u\in \mathcal{D} (A)$ and $g:=\Delta u_x+au_x$, then for any $j\ge 1$
\be
\u '''+ (a-\lambda _j ) \u ' = \g \ \textrm { in } \  L^2(-1,0) \label{AA3}
\ee
where $'=d/dx$.
For $n=0$, \eqref{8} is obvious if we pick $C_0\ge 1$.
Let us assume first that $n=1$. Note that $\u \in H^3(-1,0)$  by \eqref{AA3}. Multiplying \eqref{AA3} by $\lj (x+1)\u$, we obtain
\[
\frac{3}{2} \lj \int_{-1}^0 |\u'|^2 dx - (a-\lj ) \frac{\lj}{2} \int_{-1}^0  |\u |^ 2 dx = \lj \int_{-1}^0 (x+1) \u \g dx.
\]
Let $j_0:= \left[ \frac{\sqrt{2a}}{\pi} \right]$. Then for $j>j_0$, we have $a\le \lj/2$ and hence $|a-\lj|\lj/2\ge \lj^2/4$. 
Using
\[
\left\vert \lj \int_{-1}^0 (x+1) \u \g dx\right\vert \le \frac{\lj ^2}{8}\int_{-1}^0 |\u |^2 dx + 2\int_{-1}^0 |\g |^2 dx, 
\]
we infer that for $j>j_0$
\[
\frac{3}{2}\lj \int_{-1}^0 |\u '|^2 dx + \frac{\lj ^2}{8} \int_{-1}^0 |\u |^2 dx  \le 2\int_{-1} ^0 |\g |^2dx, 
\]
and that for $1\le j\le j_0$ 
\be
\frac{3}{2}\lj \int_{-1}^0 |\u '|^2 dx + \frac{\lj ^2}{8} \int_{-1}^0 |\u |^2 dx  \le
\Lambda \int_{-1}^0 |\u |^2 dx  + 2 \int_{-1} ^0 |\g|^2dx, 
\label{AA6}
\ee
where $\Lambda := \max_{1\le j\le j_0} \left\vert \frac{\lj ^2}{8} + (a-\lj) \frac{\lj }{2}    \right\vert$.  
Obviously, \eqref{AA6} is valid for any $j\ge 1$. 
Summing in $j$, we obtain
\be
\label{AA7}
\frac{3}{2} \int_\Omega |u_{xy} |^2 dxdy + \frac{1}{8} \int_\Omega |u_{yy}|^2dxdy \le \Lambda \Vert u\Vert ^2 + 2 \Vert g \Vert ^2. 
\ee
Dividing in \eqref{AA6} by $\lj \ge \pi ^2$ and summing in $j$, we obtain 
\be
\label{AA8}
\frac{3}{2} \int_\Omega |u_x|^2 dxdy +\frac{1}{8} \int_\Omega |u_y|^2 dxdy \le \frac{\Lambda}{\pi ^2} \Vert u\Vert ^2 + \frac{2}{\pi ^2} \Vert g\Vert ^2.
\ee
It remains to estimate $\int_\Omega | u_{xx} |^2 dxdy$. Multiplying in \eqref{AA3} by $\u '$, we obtain
\[
-\int_{-1}^0 | \u '' |^2 dx + \u ' \u '' \vert _{-1}^0 +(a-\lj ) \int_{-1}^0 |\u' |^2dx = \int_{-1}^0 \g \u 'dx,  
 \]
and hence 
\[
 \int_{-1}^0 |\u ''|^2 dx \le |\lj -a|\int_{-1}^0 |\u '|^2 dx + \Vert \g\Vert \Vert \u'\Vert + |\u'(-1)\u ''(-1)|. 
\]
We are let to estimate $\u '(-1)$ and $\u ''(-1)$.
Multiplying in \eqref{AA3}  by $\lj \u$ results in 
\[
\lj \frac{\u '(-1)^2}{2} = \lj \int_{-1}^0 \u \g dx.
\]
Combined with \eqref{AA6}, this yields
\be
\label{AA10}
\frac{\lj}{2} \u ' (-1)^2 \le \Vert \lj \u \Vert \cdot \Vert \g\Vert\le \frac{1}{4} \Vert \lj \u \Vert ^2 + \Vert \g\Vert ^2 \le 5\Vert \g\Vert ^2 + 
2\Lambda \Vert \u\Vert ^2. 
\ee
Multiplying in \eqref{AA3} by $x$, we obtain 
\[
-\int_{-1}^0 \u ''dx + x \u ''\vert _{-1}^0 + (a-\lj) \left(-\int_{-1}^0 \u dx +x\u \vert_{-1}^0  \right)
=\int_{-1}^0 x\g dx 
\] 
which yields 
\[
\u ''( -1 ) = -\u '(-1) +(a-\lj) \int_{-1}^0 \u dx +\int_{-1}^0 x\g dx, 
\]
so that 
\[
| \u ''(-1) |^2 \le 3\left( | \u '(-1) |^2 + 2\big( a^2 + | \lj |^2 \big)\Vert \u\Vert ^2  + \Vert \g \Vert ^2\right) .
\]
Using \eqref{AA6} and \eqref{AA10}, 
we conclude that $\vert \u''(-1)|^2 =O( \Vert \u\Vert ^2  + \Vert \g \Vert ^2)$. The same is true for $\Vert \u ''\Vert ^2$. Gathering together the above estimates, we arrive at
\[
\Vert u\Vert ^2_{H^2(\Omega )} \le C_1 \left( \Vert u\Vert ^2_{L^2(\Omega )} +\Vert P u\Vert ^2_{L^2(\Omega )} \right)
\]
for some constant $C_1=C_1(a)>0$. 

Let us check that ${\mathcal D} (A^n)\subset H^{2n}(\Omega )$  for $n\ge 2$. We proceed by induction on $n$. Assume 
that  ${\mathcal D} (A^p)\subset H^{2p}(\Omega )$ for $p=0,1,...,n-1$ (with $n-1\ge 1$), and pick any 
$u\in {\mathcal D} (A^n)$. Then $g=Au\in {\mathcal D}(A^{n-1}) \subset H^{2(n-1)} (\Omega )$. Let 
$h:= (-1)^{n-1}\partial _y^{ 2(n-1) }g\in L^2(\Omega )$.  Then, using Lemmas  \ref{lemma0} and 
\ref{lemma0bis}, we have that for all $j\ge 1$
\be
\label{AA21}
\lj ^{n-1} (\u '''+(a-\lj )\u ') = \h . 
\ee

Multiplying in \eqref{AA21} by $\lj ^n (x+1)\u$, we obtain 
\[
\frac{3}{2} \lj ^{2n-1} \int_{-1}^0 |\u'|^2 dx 
- (a-\lj ) \frac{\lj ^{2n-1}}{2} \int_{-1}^0  |\u |^ 2 dx = \lj ^n \int_{-1}^0 (x+1) \u \h dx.
\]
This yields 
\be
\lj ^{2n-1} \Vert \u ' \Vert ^2 + \lj ^{2n} \Vert \u \Vert  ^2 
=O( \Vert \u \Vert  ^2  + \Vert  \h \Vert ^2) , 
\label{AA22}
\ee
Multiplying in \eqref{AA21}  by $\lj ^n \u$ gives 
\[
\lj ^{2n-1} \frac{\u '(-1)^2}{2} = \lj ^n \int_{-1}^0 \u \h dx
\]
and 
\[
\lj ^{2n-1} | \u ' (-1)|^2 = O( \Vert \u \Vert  ^2  + \Vert  \h \Vert ^2).
\]
From 
\[
\lj ^{n-1}\u ''( -1 ) = -\lj ^{n-1} \u '(-1) +(a-\lj) \lj^{n-1} \int_{-1}^0 \u dx +\int_{-1}^0 x\h dx, 
\]
we infer that 
\[
 \lj ^{2n-2}  | \u ''(-1) |^2 = O( \Vert \u \Vert  ^2  + \Vert  \h \Vert ^2). 
  \]
It follows from 
\[
- \lj ^{2n-2} \int_{-1}^0 | \u '' |^2 dx + \lj ^{2n-2}\u ' \u '' \vert _{-1}^0 +(a-\lj )\lj ^{2n-2} \int_{-1}^0 |\u' |^2dx =  \lj ^{n-1}\int_{-1}^0 \h \u 'dx,  
\]
that 
\be
\lj ^{2n-2} \Vert \u '' \Vert ^2 
=O( \Vert \u \Vert  ^2  + \Vert  \h \Vert ^2) . 
\label{AA23}
\ee
So far, we have proved that 
\[
\sum_{j=1}^\infty 
\left(
\lambda _j ^{2n} \Vert \hat u_j \Vert ^2  
+  \lambda _j ^{2n-1} \Vert \hat u_j'\Vert ^2 
+ \lambda _j ^{2n-2}  \Vert \hat u_j ''\Vert ^2 
\right) <+\infty.
\]
Using Lemma \ref{lemma0bis}, this gives that $\partial _y^{2n}u, \partial _y^{2n-1}\partial _x u$, and $\partial _y^{2n-2}\partial _x^2u$ belong to
 $L^2(\Omega)$. 
For the other derivatives of order $2n$, we apply the operator $\partial _x^{2k}$ (for $k\in \N $ with $2k+3\le 2n$) to each term in \eqref{AA3} to obtain
\[
\u ^{(2k+3)}+(a-\lj )\u ^{(2k+1)}  = \g ^{(2k)} . 
\]
This yields
\[
\lj ^{(2n-3-2k)} \Vert \u ^{(2k+3)} \Vert ^2 = O (\Vert \u \Vert ^2 + \Vert \h \Vert ^2 + \lj ^{2n-3-2k} \Vert \g ^{(2k)}\Vert ^2).
\]
On the other hand, \eqref{AA3} gives by differentiation with respect to  $x$ that 
\[
\u ^{(4)}+(a-\lj )\u ''  = \g ',
\]
and we obtain in a similar way that 
\[
\lj ^{(2n-4-2k)} \Vert \u ^{(2k+4)} \Vert ^2 = O (\Vert \u \Vert ^2 + \Vert \h \Vert ^2 + \lj ^{2n-4-2k} \Vert \g ^{(2k+1)}\Vert ^2).
\]
for $k\in \N$ with $2k+4\le 2n$.  
Thus we conclude that 
 \[
 \sum_{q=0}^{2n} \sum_{j=1}^\infty \lambda _j ^{(2n-q)} \Vert \hat u_j ^{(q)}\Vert ^2 < +\infty .
 \]
 Using Lemma  \ref{lemma0bis}, we infer that for $q\in \{ 0, ..., 2n \}$,  $\partial _x ^q u\in L^2(-1,0, H^{2n-q}(0,1))$, and hence that 
 $\partial _y^{2n-q}\partial _x^q u\in L^2(\Omega)$. We also have that $\partial _y^{2n-1-q}\partial _x^q u\in L^2(\Omega)$ for $q\in \{ 0 , ..., 2n-1\}$. 
 Taking into account the fact that $u\in {\mathcal D} (A^{n-1})\subset H^{2(n-1)} (\Omega )$,
 we conclude that $u\in H^{2n}(\Omega )$. The proof of  the  inclusion ${\mathcal D} (A^n)\subset H^{2n}(\Omega)$ is complete.  
 
It remains to prove  that the constant in the r.h.s. of \eqref{8} is indeed of the form $B^n$.
This will require a series of lemmas. 

\begin{lemma}
\label{Lemma2.2}
For any $\varepsilon_0 >0$, there exists  a constant $K=K(\varepsilon _0) >0$ such that 
for all $\varepsilon \in (0, \varepsilon _0)$ and all $f\in H^2(-1,0)$, 
\begin{equation}
\label{lemadams}
\int_{-1}^{0}|f^{\prime}(t)|^{2}dt\leq K\varepsilon \int_{-1}^{0}|f^{\prime\prime}(t)|^{2}dt+K\varepsilon^{-1} \int_{-1}^{0}|f(t)|^{2}dt.
\end{equation}
\end{lemma}
Lemma \ref{Lemma2.2} is a direct consequence of \cite[Lemma 4.10]{adams}  (which is concerned with twice continuously functions)
by density of $C^2([-1,0])$ in $H^2(-1,0)$.

For any $j\in \mathbb{N^{*}}$, we define the operator $P_j$ by 
$$P_{j}f :=f^{\prime\prime\prime}-(\lambda_{j}-a)f^{\prime}, \qquad \forall f\in H^{3}(-1,0). $$
\begin{lemma}\label{Lemma2.3}
There exists a constant $C_{1}\geq1$ such that
\begin{equation}\label{1-2}
\|f\|^{2}_{H^{2n}(-1,0)}\leq C_{1}^{n}\sum\limits_{i=0}\limits^{n}\lambda_{j}^{2n-2i}\|P_{j}^{i}f\|^{2},~~\forall n\in\mathbb{N},~\forall j\in\mathbb{N}^{*},~\forall f\in H^{3n}(-1,0).
\end{equation}
\end{lemma}
\begin{proof}
For $n=0$, \eqref{1-2} is obvious.
For $n=1$, it follows from the definition of $P_{j}$ and Lemma \ref{Lemma2.2} that
\begin{equation*}
\begin{split}
\|f\|^{2}_{H^{2}(-1,0)}=&\|f\|^{2}+\|f^{\prime}\|^{2}+\|f^{\prime\prime}\|^{2}\\
\leq& C(\|f\|^{2}+\|f^{\prime\prime}\|^{2})\\
\leq& C(\|f\|^{2}+\frac{1}{\lambda_{j}}\|f^{\prime\prime\prime}\|^{2}+\lambda_{j}\|f^{\prime}\|^{2})\\
\leq& C(\|f\|^{2}+\frac{1}{\lambda_{j}}\|P_{j}f\|^{2}+\lambda_{j}\|f^{\prime}\|^{2})\\
\leq& C(\|f\|^{2}+\|P_{j}f\|^{2})+\frac{1}{2}\|f^{\prime\prime}\|^{2}+C\lambda^{2}_{j}\|f\|^{2}\\
\leq& C(\lambda^{2}_{j}\|f\|^{2}+\|P_{j}f\|^{2})+\frac{1}{2}\|f^{\prime\prime}\|^{2}.
\end{split}
\end{equation*}
This shows that we can find a constant $C_{2}\geq1$ such that
\begin{equation*}
\|f\|^{2}_{H^{2}(-1,0)}\leq C_{2}(\lambda^{2}_{j}\|f\|^{2}+\|P_{j}f\|^{2}).
\end{equation*}

Let us prove \eqref{1-2} for $n\ge 2$ by induction on $n$.  Assume \eqref{1-2} to be true for $n-1\geq0$. It follows that
\begin{equation*}
\begin{split}
\|f\|_{H^{2n}(-1,0)}^{2}=&\|f\|_{H^{2n-2}(-1,0)}^{2}+\|f^{(2n-1)}\|^{2}+\|f^{(2n)}\|^{2}\\
\leq& \|f\|_{H^{2n-2}(-1,0)}^{2}+\|f^{(2n-2)}\|_{H^{2}(-1,0)}^{2}\\
\leq& \|f\|_{H^{2n-2}(-1,0)}^{2}+C_{2}(\lambda^{2}_{j}\|f^{(2n-2)}\|^{2}+\|P_{j}f^{(2n-2)}\|^{2})\\
\leq& 2C_{2}\lambda^{2}_{j}\|f\|_{H^{2n-2}(-1,0)}^{2}+C_{2}\|P_{j}f\|_{H^{2n-2}(-1,0)}^{2}\\
\leq& 2C_{2}\lambda^{2}_{j}C_{1}^{n-1}\sum\limits_{i=0}\limits^{n-1}\lambda_{j}^{2n-2-2i}\|P_{j}^{i}f\|^{2}
+C_{2}C_{1}^{n-1}\sum\limits_{i=0}\limits^{n-1}\lambda_{j}^{2n-2-2i}\|P_{j}^{i+1}f\|^{2}\\
\leq& 2C_{2}C_{1}^{n-1}\sum\limits_{i=0}\limits^{n-1}\lambda_{j}^{2n-2i}\|P_{j}^{i}f\|^{2}
+C_{2}C_{1}^{n-1}\sum\limits_{i=1}\limits^{n}\lambda_{j}^{2n-2i}\|P_{j}^{i}f\|^{2}\\
\leq& 3C_{2}C_{1}^{n-1}\sum\limits_{i=0}\limits^{n}\lambda_{j}^{2n-2i}\|P_{j}^{i}f\|^{2}.
\end{split}
\end{equation*}
If we pick $C_{1}=3C_{2}$, \eqref{1-2} is true for $n$. 
\end{proof}

\begin{lemma}\label{Lemma2.4}
There exists a positive constant $C_{3}$ such that
\begin{equation*}
\|u\|_{H^{2n}(\Omega)}^{2}\leq C_{3}\sum\limits_{m=0}^{n}\sum\limits_{k=0}^{m}\|\partial_{x}^{2k}\partial_{y}^{2m-2k}u\|_{L^{2}(\Omega)}^{2},~~~\forall n\in\mathbb{N},~\forall u\in \mathcal{D}(A^{n}).
\end{equation*}
\end{lemma}
\begin{proof}
For any $p\in \mathbb{N}$, we set
\begin{equation*}
I_{p} := \sum\limits_{a,b\in\mathbb{N},~a+b=p}\|\partial_{x}^{a}\partial_{y}^{b}u\|_{L^{2}(\Omega)}^{2}.
\end{equation*}
Decompose $u$ as
\begin{equation}\label{1-6}
u(x,y)=\sum\limits_{j=1}^{\infty}\hat{u}_{j}(x)e_{j}(y).
\end{equation}

Let us go back to the proof of Lemma \ref{Lemma2.4}. Pick any $u\in {\mathcal D}(A^n)$, for some $n\in {\mathbb N}$. 
Using Lemma \ref{lemma0} and  applying Lemma \ref{lemma0bis} to the functions $\partial_{x}^{2m+1-k}u(x,.)$ for $0\le m\le n-1$, $0\le k\le 2m+1$, and $x\in (-1,0)$,  
we obtain that
\begin{equation*}
\begin{split}
I_{2m+1}=&\sum\limits_{k=0}^{2m+1}\|\partial_{x}^{2m+1-k}\partial_{y}^{k}u\|_{L^{2}(\Omega)}^{2}\\
=&\sum\limits_{k=0}^{2m+1}\sum\limits_{j=1}^{\infty}\lambda_{j}^{k}\|\hat{u}_{j}^{(2m+1-k)}\|^{2}\\
=&\sum\limits_{j=1}^{\infty}\|\hat{u}_{j}^{(2m+1)}\|^{2}+\sum\limits_{k=1}^{2m+1}\sum\limits_{j=1}^{\infty}\lambda_{j}^{k}\|\hat{u}_{j}^{(2m+1-k)}\|^{2}\\
\leq& \sum\limits_{j=1}^{\infty}\lambda_{j}\|\hat{u}_{j}^{(2m+1)}\|^{2}+\frac{1}{2}\sum\limits_{k=1}^{2m+1}\sum\limits_{j=1}^{\infty}\lambda_{j}^{k-1}\|\hat{u}_{j}^{(2m+1-k)}\|^{2}
+\frac{1}{2}\sum\limits_{k=1}^{2m+1}\sum\limits_{j=1}^{\infty}\lambda_{j}^{k+1}\|\hat{u}_{j}^{(2m+1-k)}\|^{2}\\
=&\|\partial_{x}^{2m+1}\partial_{y}u\|_{L^{2}(\Omega)}^{2}+\frac{1}{2}\sum\limits_{k=1}^{2m+1}\|\partial_{x}^{2m+1-k}\partial_{y}^{k-1}u\|_{L^{2}(\Omega)}^{2}
+\frac{1}{2}\sum\limits_{k=1}^{2m+1}\|\partial_{x}^{2m+1-k}\partial_{y}^{k+1}u\|_{L^{2}(\Omega)}^{2}\\
\leq&\frac{3}{2}I_{2m+2}+\frac{1}{2}I_{2m},
\end{split}
\end{equation*}
where we used Young's estimate. 
Thus, we have
\begin{equation}\label{1-3}
\begin{split}
\|u\|_{H^{2n}(\Omega)}^{2}=&\sum\limits_{m=0}^{n}I_{2m}+\sum\limits_{m=0}^{n-1}I_{2m+1}\\
\leq&\sum\limits_{m=0}^{n}I_{2m}+\sum\limits_{m=0}^{n-1}(\frac{3}{2}I_{2m+2}+\frac{1}{2}I_{2m})\\
\leq&3\sum\limits_{m=0}^{n}I_{2m}.
\end{split}
\end{equation}
Next, we consider $I_{2m}$. For $m=0$, $I_{0}=\|u\|_{L^{2}(\Omega)}^{2}$. 
For $m\geq1$, we have
\begin{equation}\label{1-4}
I_{2m}=\sum\limits_{k=0}^{m}\|\partial_{x}^{2k}\partial_{y}^{2m-2k}u\|_{L^{2}(\Omega)}^{2}+\sum\limits_{k=0}^{m-1}\|\partial_{x}^{2k+1}\partial_{y}^{2m-2k-1}u\|_{L^{2}(\Omega)}^{2},
\end{equation}
and it remains to estimate the second term in the r.h.s. of (\ref{1-4}). Applying Lemma \ref{Lemma2.2}, we obtain 
\begin{equation}\label{1-5}
\begin{split}
\sum\limits_{k=0}^{m-1}\|\partial_{x}^{2k+1}\partial_{y}^{2m-2k-1}u\|_{L^{2}(\Omega)}^{2}=&\sum\limits_{k=0}^{m-1}\sum\limits_{j=1}^{\infty}\lambda_{j}^{2m-2k-1}\|\hat{u}_{j}^{(2k+1)}\|^{2}\\
\leq& C\Big(\sum\limits_{k=0}^{m-1}\sum\limits_{j=1}^{\infty}\lambda_{j}^{2m-2k-2}\|\hat{u}_{j}^{(2k+2)}\|^{2}+
\sum\limits_{k=0}^{m-1}\sum\limits_{j=1}^{\infty}\lambda_{j}^{2m-2k}\|\hat{u}_{j}^{(2k)}\|^{2}\Big)\\
=&C\Big(\sum\limits_{k=1}^{m}\sum\limits_{j=1}^{\infty}\lambda_{j}^{2m-2k}\|\hat{u}_{j}^{(2k)}\|^{2}+
\sum\limits_{k=0}^{m-1}\sum\limits_{j=1}^{\infty}\lambda_{j}^{2m-2k}\|\hat{u}_{j}^{(2k)}\|^{2}\Big)\\
\leq&C\sum\limits_{k=0}^{m}\sum\limits_{j=1}^{\infty}\lambda_{j}^{2m-2k}\|\hat{u}_{j}^{(2k)}\|^{2}\\
=&C\sum\limits_{k=0}^{m}\|\partial_{x}^{2k}\partial_{y}^{2m-2k}u\|_{L^{2}(\Omega)}^{2}.
\end{split}
\end{equation}
Combining \eqref{1-3}-\eqref{1-5}, the conclusion of Lemma \ref{Lemma2.4} follows.  
\end{proof}
\begin{lemma}\label{Lemma2.5}
There exists a constant $C_{4}\geq1$ such that
\begin{equation}\label{1-7}
\lambda_{j}^{2m}\|P^{i}_{j}\hat{u}_{j}\|^{2}\leq C_{4}^{m}\sum\limits_{l=0}^{m} \left( \begin{array}{c}  m\\ l \end{array} \right)\|P_{j}^{i+l}\hat{u}_{j}\|^{2},~~~\forall m,i\in \mathbb{N}, \forall j\in\mathbb{N}^{*},\forall u\in\mathcal{D}(A^{m+i}),
\end{equation}
where $\hat{u}_{j}$ is the Fourier coefficients of $u$ as in \eqref{1-6}.
\end{lemma}
\begin{proof} The proof is by induction on $m$. 
For $m=0$, \eqref{1-7} is obvious for any $C_{4}\geq1$.

For $m=1$ and $u\in\mathcal{D}(A^{1+i})$, we have that $P^{i}u\in\mathcal{D}(A)$ and, by \cite[Lemma 4.1]{STW},
\begin{equation*}
(P^{i}u)(x,y)=\sum\limits_{j=1}^{\infty}(P_{j}^{i}\hat{u}_{j})(x)e_{j}(y),
\end{equation*}
where the function $P_{j}^{i}\hat{u}_{j}$ satisfies
for each $j\in\mathbb{N}^{*}$ 
\begin{eqnarray}\label{1-8}
\begin{array}{l}
\left\{
\begin{array}{lll}(P_{j}^{i}\hat{u}_{j})^{\prime\prime\prime}-(\lambda_{j}-a)(P_{j}^{i}\hat{u}_{j})^{\prime}=P_{j}^{i+1}\hat{u}_{j},
\\(P_{j}^{i}\hat{u}_{j})(-1)=(P_{j}^{i}\hat{u}_{j})(0)=(P_{j}^{i}\hat{u}_{j})^{\prime}(0)=0.
\end{array}
\right.
\end{array}
\begin{array}{lll}x\in(-1,0),
\\ ~
\end{array}
\end{eqnarray}
Multiplying the first equation in \eqref{1-8} by $\lambda_{j}(x+1)P_{j}^{i}\hat{u}_{j}$ and integrating over $(-1,0)$ results in 
\begin{equation*}
\frac{3}{2}\lambda_{j}\int_{-1}^{0}|(P_{j}^{i}\hat{u}_{j})^{\prime}|^{2}dx+(\lambda_{j}-a)\frac{\lambda_{j}}{2}\int_{-1}^{0}|P_{j}^{i}\hat{u}_{j}|^{2}dx
=\lambda_{j}\int_{-1}^{0}(x+1)(P_{j}^{i}\hat{u}_{j})(P_{j}^{i+1}\hat{u}_{j})dx.
\end{equation*}
After some elementary calculations, we can find a constant $C_{4}=C_4(a)\geq1$ such that
\begin{equation*}
\lambda_{j}\|(P_{j}^{i}\hat{u}_{j})^{\prime}\|^{2}+\lambda_{j}^{2}\|P_{j}^{i}\hat{u}_{j}\|^{2}\leq C_{4}(\|P_{j}^{i}\hat{u}_{j}\|^{2}+\|P_{j}^{i+1}\hat{u}_{j}\|^{2}).
\end{equation*}
Therefore, \eqref{1-7} holds for $m=1$.
Pick now any $m\geq2$, and assume that \eqref{1-7} is true for $m-1\geq0$. For any $u\in\mathcal{D}(A^{m+i})$, we have
\begin{equation*}
\lambda_{j}^{2m}\|P^{i}_{j}\hat{u}_{j}\|^{2}=\lambda_{j}^{2}\lambda_{j}^{2m-2}\|P^{i}_{j}\hat{u}_{j}\|^{2}
\leq \lambda_{j}^{2}C_{4}^{m-1}\sum\limits_{l=0}^{m-1} \left( \begin{array}{c}  m-1\\ l \end{array} \right)\|P_{j}^{i+l}\hat{u}_{j}\|^{2}.
\end{equation*}
Since $u\in\mathcal{D}(A^{m+i})$, for any $l=0,1,...,m-1$, system \eqref{1-8} is satisfied with $P_{j}^{i+l}\hat{u}_{j}$ substituted to 
$P_j^i \hat u_j$, and it follows as above that 
\begin{equation*}
\lambda_{j}^{2}\|P_{j}^{i+l}\hat{u}_{j}\|^{2}\leq C_{4}(\|P_{j}^{i+l}\hat{u}_{j}\|^{2}+\|P_{j}^{i+l+1}\hat{u}_{j}\|^{2}).
\end{equation*}
We infer that
\begin{equation*}
\begin{split}
\lambda_{j}^{2m}\|P^{i}_{j}\hat{u}_{j}\|^{2}\leq&C_{4}^{m}\sum\limits_{l=0}^{m-1} \left( \begin{array}{c}  m-1\\ l \end{array} \right)
(\|P_{j}^{i+l}\hat{u}_{j}\|^{2}+\|P_{j}^{i+l+1}\hat{u}_{j}\|^{2})\\
=&C_{4}^{m}(\|P_{j}^{i}\hat{u}_{j}\|^{2}+\sum\limits_{l=1}^{m-1} \left( \begin{array}{c}  m-1\\ l \end{array} \right)\|P_{j}^{i+l}\hat{u}_{j}\|^{2}+\sum\limits_{l=1}^{m-1} \left( \begin{array}{c}  m-1\\ l-1 \end{array} \right)\|P_{j}^{i+l}\hat{u}_{j}\|^{2}+\|P_{j}^{i+m}\hat{u}_{j}\|^{2})\\
=&C_{4}^{m}(\|P_{j}^{i}\hat{u}_{j}\|^{2}+\sum\limits_{l=1}^{m-1} \left( \begin{array}{c}  m\\ l \end{array} \right)\|P_{j}^{i+l}\hat{u}_{j}\|^{2}+\|P_{j}^{i+m}\hat{u}_{j}\|^{2})\\
=&C_{4}^{m} \sum\limits_{l=0}^{m} \left( \begin{array}{c}  m\\ l \end{array} \right)\|P_{j}^{i+l}\hat{u}_{j}\|^{2}
\end{split}
\end{equation*}
where we used Pascal's Rule. 
The proof of Lemma \ref{Lemma2.5} is achieved.
\end{proof}
We are in a position to complete the proof of Proposition \ref{P1}. 
The estimate (\ref{8}) is obvious for $n=0$. Let $n\ge 1$. Using Lemmas \ref{Lemma2.3}, \ref{Lemma2.4}, and 
\ref{Lemma2.5}, we obtain that
\begin{equation*}
\begin{split}
\|u\|^{2}_{H^{2n}(\Omega)}\leq &C_{3}\sum\limits_{m=0}^{n}\sum\limits_{k=0}^{m}\|\partial_{x}^{2k}\partial_{y}^{2m-2k}u\|_{L^{2}(\Omega)}^{2}\\
=&C_{3}\sum\limits_{m=0}^{n}\sum\limits_{k=0}^{m}\sum\limits_{j=1}^{\infty}\lambda_{j}^{2m-2k}\|\hat{u}^{(2k)}_{j}\|^{2}\\
\leq&C_{3}\sum\limits_{m=0}^{n}\sum\limits_{k=0}^{m}\sum\limits_{j=1}^{\infty}\lambda_{j}^{2m-2k}C_{1}^{k}\sum\limits_{i=0}\limits^{k}\lambda_{j}^{2k-2i}\|P_{j}^{i}\hat{u}_{j}\|^{2}\\
\leq&C_{3}C_{1}^{n}\sum\limits_{m=0}^{n}\sum\limits_{k=0}^{m}\sum\limits_{j=1}^{\infty}\sum\limits_{i=0}\limits^{k}\lambda_{j}^{2m-2i}\|P_{j}^{i}\hat{u}_{j}\|^{2}.
\end{split}
\end{equation*}
Using the fact that $i\le k\le m\le n$ in the sum above, we obtain
\begin{equation*}
\begin{split}
\|u\|^{2}_{H^{2n}(\Omega)}
\leq&C_{3}C_{1}^{n}\sum\limits_{m=0}^{n}\sum\limits_{k=0}^{n}\sum\limits_{j=1}^{\infty}\sum\limits_{i=0}\limits^{n}\lambda_{j}^{2n-2i}\|P_{j}^{i}\hat{u}_{j}\|^{2}\\
\leq&C_{3}C_{1}^{n}(n+1)^{2}\sum\limits_{j=1}^{\infty}\sum\limits_{i=0}\limits^{n}\lambda_{j}^{2n-2i}\|P_{j}^{i}\hat{u}_{j}\|^{2}\\
\leq&C_{3}C_{1}^{n}(n+1)^{2} \sum\limits_{j=1}^{\infty}\sum\limits_{i=0}\limits^{n}C_{4}^{n-i}\sum\limits_{l=0}^{n-i}\left( \begin{array}{c}  n-i\\ l \end{array} \right)\|P_{j}^{i+l}\hat{u}_{j}\|^{2}\\
\leq&C_{3}C_{1}^{n}(n+1)^{2}C_{4}^{n}2^{n}\sum\limits_{j=1}^{\infty}\sum\limits_{i=0}\limits^{n}\sum\limits_{l=i}^{n}\|P_{j}^{l}\hat{u}_{j}\|^{2}\\
\leq&C_{3}C_{1}^{n}(n+1)^{2}C_{4}^{n}2^{n}\sum\limits_{j=1}^{\infty}\sum\limits_{i=0}\limits^{n}\sum\limits_{l=0}^{n}\|P_{j}^{l}\hat{u}_{j}\|^{2}\\
\leq&C_{3}C_{1}^{n}(n+1)^{3}C_{4}^{n}2^{n}\sum\limits_{l=0}\limits^{n}\sum\limits_{j=1}^{\infty}\|P_{j}^{l}\hat{u}_{j}\|^{2}\\
\leq &B^{n}\sum\limits_{l=0}\limits^{n}\|P^{l}u\|_{L^{2}(\Omega)}^{2}
\end{split}
\end{equation*}
with $B:=16C_1C_3C_4$. Indeed, it is easy to see that $(n+1)^3\le 8^n$ for all $n\in {\mathbb N}$.
The proof of Proposition \ref{P1} is achieved. 
\end{proof}

Recall that $\lambda_{j}=(j\pi)^{2}$ for $j\ge 1$.
For any $j\geq1$, we consider a sequence of generating functions $g_{i,j}~(i\geq0)$, where $g_{0,j}$ is the solution of the Cauchy problem
\begin{eqnarray}\label{2}
\begin{array}{l}
\left\{
\begin{array}{lll}g_{0,j}^{\prime\prime\prime}(x)-(\lambda_{j}-a)g_{0,j}^{\prime}(x)=0,
\\g_{0,j}(0)=g_{0,j}^{\prime}(0)=0,~~~g_{0,j}^{\prime\prime}(0)=1,
\end{array}
\right.
\end{array}
\begin{array}{lll}x\in(-1,0),
\\~
\end{array}
\end{eqnarray}
while $g_{i,j}$ for $i\geq1$ is defined inductively as the solution of the Cauchy problem
\begin{eqnarray}\label{3}
\begin{array}{l}
\left\{
\begin{array}{lll}g_{i,j}^{\prime\prime\prime}(x)-(\lambda_{j}-a)g_{i,j}^{\prime}(x)= -g_{i-1,j} (x),
\\g_{i,j}(0)=g_{i,j}^{\prime}(0)=g_{i,j}^{\prime\prime}(0)=0.
\end{array}
\right.
\end{array}
\begin{array}{lll}x\in(-1,0),
\\~
\end{array}
\end{eqnarray}
\begin{proposition}\label{P2}
For any $i\geq0,j\geq1$ and $x\in[-1,0]$, we have
\begin{equation}\label{4}
|g_{i,j}(x)|\leq
    e^{\sqrt{\lambda_{j}}}\frac{3^{i}i!}{(3i+2)!}.
\end{equation}
\end{proposition}
\begin{proof}
It follows from \eqref{2} and \eqref{3} that
\begin{equation*}
\begin{split}
g_{i,j}(x)=&-\int_{0}^{x}g_{0,j}(x-\xi)g_{i-1,j}(\xi)d\xi\\
=&-\int_{0}^{x}g_{0,j}^{\prime\prime}(x-\xi)\Big(\int_{0}^{\xi}
(\int_{0}^{\zeta}g_{i-1,j}(\sigma)d\sigma ) d\zeta\Big) d\xi,~~~i,j\geq1.
\end{split}
\end{equation*}
\noindent (1) if $\lambda_{j}\leq a$, it is not difficult to obtain that
\begin{equation*}
g_{0,j}(x)=
\begin{cases}
    \displaystyle\frac{1}{a-\lambda_{j}}(1-\cos(\sqrt{a-\lambda_{j}}x)), \quad& \lambda_{j}< a; \\
    \displaystyle\frac{1}{2}x^{2}, \quad & \lambda_{j}= a,
\end{cases}
\end{equation*}
this implies
\begin{equation*}
0\leq g_{0,j}(x)\leq \frac{x^{2}}{2},~~\forall~j\geq1,~x\in[-1,0].
\end{equation*}
Then it follows from \cite[Lemma 2.1]{MRRR} that
\begin{equation*}
 |g_{i,j}(x)|\leq \frac{|x|^{3i+2}}{(3i+2)!}\leq e^{\sqrt{\lambda_{j}}}\frac{3^{i}i!}{(3i+2)!}, \ \ \forall i\geq0, \ \forall j\geq
 1, \ \forall x\in[-1,0].
\end{equation*}\par
\noindent (2) if $\lambda_{j}> a$, we claim that
\begin{equation}\label{5}
g_{i,j}(x)\leq
    \cosh(\sqrt{\lambda_{j}-a}x)\frac{(-x)^{3i+2}3^{i}i!}{(3i+2)!}, \ \ \forall i\geq0, \ \forall j\geq1, \ \forall x\in[-1,0]
\end{equation}
which implies \eqref{4}.\par
Let us prove \eqref{5} by induction on $i$.
For $i=0$,
\begin{equation*}
\begin{split}
0\le g_{0,j}(x)=&\frac{1}{\lambda_{j}- a}(\cosh(\sqrt{\lambda_{j}- a}x)-1)\\
=&\sum\limits_{q=1}\limits^{\infty}\frac{(\lambda_{j}- a)^{q-1}x^{2q}}{(2q)!}\\
\leq&\sum\limits_{q=1}\limits^{\infty}\frac{(\lambda_{j}- a)^{q-1}x^{2q-2}}{(2q-2)!}\frac{x^{2}}{2!}\\
=&\cosh(\sqrt{\lambda_{j}-a}x)\frac{x^{2}}{2!},
\end{split}
\end{equation*}
so that \eqref{5} is true for $i=0$.\par
Assume now that \eqref{5} is true for $i-1\geq0$. We can deduce that for $x\in [-1,0]$
\begin{equation*}
\begin{split}
|g_{i,j}(x)|
\leq&-\int_{0}^{x}g_{0,j}^{\prime\prime}(x-\xi)\Big(\int_{0}^{\xi}
\big( \int_{0}^{\zeta}   |g_{i-1,j}(\sigma)|d\sigma  \big) d\zeta\Big) d\xi\\
\leq&-\int_{0}^{x}\sum\limits_{p=0}\limits^{\infty}\frac{(\lambda_{j}- a)^{p}(x-\xi)^{2p}}{(2p)!}
\Big(\int_{0}^{\xi} \big( \int_{0}^{\zeta}3^{i-1}(i-1)!\sum\limits_{q=0}\limits^{\infty}\frac{(\lambda_{j}- a)^{q}(-\sigma)^{3i-1+2q}}{(2q)!(3i-1)!}d\sigma \big) d\zeta\Big) d\xi\\
=&-3^{i-1}(i-1)!\int_{0}^{x}\sum\limits_{p=0}\limits^{\infty}\frac{(\lambda_{j}- a)^{p}(x-\xi)^{2p}}{(2p)!}
\sum\limits_{q=0}\limits^{\infty}\frac{(\lambda_{j}- a)^{q}(-\xi)^{3i+1+2q}}{(2q)!(3i-1)!(3i+2q)(3i+2q+1)}d\xi\\
=&-3^{i-1}(i-1)!\int_{0}^{x}\sum\limits_{p=0}\limits^{\infty}\sum\limits_{q=0}\limits^{\infty}\frac{(\lambda_{j}- a)^{p+q}(x-\xi)^{2p}(-\xi)^{3i+1+2q}}
{(2p)!(2q)!(3i-1)!(3i+2q)(3i+2q+1)}d\xi.
\end{split}
\end{equation*}
Then, integrating by parts $2p$ times, we obtain
\begin{equation*}
\begin{split}
|g_{i,j}(x)|\leq&-3^{i-1}(i-1)!\int_{0}^{x}\sum\limits_{p=0}\limits^{\infty}\sum\limits_{q=0}\limits^{\infty}\frac{(\lambda_{j}- a)^{p+q}(-\xi)^{3i+1+2q+2p}(3i+2q-1)!}
{(2q)!(3i-1)!(3i+1+2q+2p)!}d\xi\\
=&3^{i-1}(i-1)!\sum\limits_{p=0}\limits^{\infty}\sum\limits_{q=0}\limits^{\infty}\frac{(\lambda_{j}- a)^{p+q}(-x)^{3i+2+2q+2p}(3i+2q-1)!}
{(2q)!(3i-1)!(3i+2+2q+2p)!} \cdot
\end{split}
\end{equation*}\par
Next, we will show that
\begin{equation}\label{6}
\frac{3^{i-1}(i-1)!(3i+2q-1)!}{(2q)!(3i-1)!(3i+2+2q+2p)!}\leq\frac{3^{i}i!}{p+q+1}\frac{1}{(2p+2q)!(3i+2)!}~~~\forall~p,q\geq0,~i\geq1.
\end{equation}
It is easy to see that \eqref{6} is equivalent to
\begin{equation}\label{7}
\begin{split}
\frac{(3i+2q-1)!}{(2q)!(3i-1)!}\leq&\frac{3i}{p+q+1}\frac{(3i+2+2q+2p)!}{(2p+2q)!(3i+2)!}\\
=&6i\frac{(2p+2q+1)(2p+2q+3)(2p+2q+4)\cdots(2p+2q+3i+2)}{(3i+2)!} \cdot
\end{split}
\end{equation}
Since the left hand side of \eqref{7} is independent of $p$ and the right hand side of \eqref{7} is increasing 
in $p$, we only need to prove \eqref{6} for $p=0$, namely, we need to show that
\begin{equation*}
\frac{(3i+2q-1)!}{(3i-1)!}\leq\frac{3i}{q+1}\frac{(3i+2+2q)!}{(3i+2)!}~~~\forall  q\geq 0,\ \forall i\geq1,
\end{equation*}
this is obvious due to the fact that
\begin{equation*}
\begin{split}
\frac{(3i+2)!(3i+2q-1)!}{3i(3i-1)!(3i+2+2q)!}=&\frac{(3i+1)(3i+2)}{(3i+2q)(3i+2q+1)(3i+2q+2)}\\
\leq&\frac{1}{3i+2q}\\
\leq&\frac{1}{q+1} \cdot
\end{split}
\end{equation*}\par
Applying \eqref{6}, we infer that
\begin{equation*}
\begin{split}
|g_{i,j}(x)|\leq&\frac{(-x)^{3i+2}3^{i}i!}{(3i+2)!}\sum\limits_{p=0}\limits^{\infty}\sum\limits_{q=0}\limits^{\infty}\frac{(\lambda_{j}- a)^{p+q}x^{2p+2q}}
{(p+q+1)(2p+2q)!}\\
=&\frac{(-x)^{3i+2}3^{i}i!}{(3i+2)!}\sum\limits_{k=0}\limits^{\infty}\frac{(\lambda_{j}- a)^{k}x^{2k}}
{(2k)!}\\
=&\cosh(\sqrt{\lambda_{j}-a}x)\frac{(-x)^{3i+2}3^{i}i!}{(3i+2)!},
\end{split}
\end{equation*}
where we have used the fact that for any function $f:~\mathbb{N}\rightarrow\mathbb{R}_{+}$, it holds
\begin{equation*}
\sum\limits_{p=0}\limits^{\infty}\sum\limits_{q=0}\limits^{\infty}f(p+q)=\sum\limits_{k=0}\limits^{\infty}(k+1)f(k).
\end{equation*}\par
This ends the proof of Proposition \ref{P2}.
\end{proof}\par
\begin{remark}
Compared with the result in \cite[Lemma 2.1]{MRRR}, it seems that a more natural estimate of $g_{i,j}$ is
\begin{equation*}
|g_{i,j}(x)|\leq
    \cosh(\sqrt{\lambda_{j}-a})\frac{R^{i}(-x)^{3i+2}}{(3i+2)!}
\end{equation*}
for some constant $R>0$. According to the proof of Proposition \ref{P2}, to prove this result, we need to obtain that
\begin{equation*}
\frac{(3i+2q-1)!}{(3i-1)!}\leq\frac{R}{q+1}\frac{(3i+2+2q)!}{(3i+2)!}~~~\forall  q\geq0, \ \forall i\geq1.
\end{equation*}
This is equivalent to
\begin{equation*}
\frac{(q+1)(3i)(3i+1)(3i+2)}{(3i+2q)(3i+2q+1)(3i+2q+2)}\leq R~~~\forall  q\geq0, \ \forall i\geq1.
\end{equation*}
However, this is impossible if we pick $q=3i$.
\end{remark}\par
Using Proposition \ref{P2}, we can obtain the following corollary which will be used in the proof of the main results.
\begin{corollary}\label{C1}
For any $i\geq0,j\geq1$ and $x\in[-1,0]$, we have
\begin{equation}\label{18}
|g_{i,j}(x)|\leq
    Ce^{\sqrt{\lambda_{j}}}\frac{1}{(2i)!},
\end{equation}
where the constant $C$ is independent of $i$ and $j$.
\end{corollary}
\begin{proof}
By Stirling's formula $i!\sim(i/e)^{i}\sqrt{2\pi i}$, and it follows from \eqref{4} that for $i\ge 1$ and $j\ge 1$  we have
\begin{equation*}
\begin{split}
|g_{i,j}(x)|\leq &e^{\sqrt{\lambda_{j}}}\frac{3^{i}i!}{(3i+2)!}\\
\leq & Ce^{\sqrt{\lambda_{j}}}\frac{3^{i}i!}{(3i+1)(3i+2)\frac{3^{3i}}{2^{2i}}\frac{\sqrt{6\pi i}}{\sqrt{2\pi i}\sqrt{4\pi i}}(2i)!i!}\\
\leq & Ce^{\sqrt{\lambda_{j}}}\frac{1}{(2i)!} \cdot
\end{split}
\end{equation*}
\end{proof}
\par
\section{Null controllability}
\begin{proposition}\label{P3}
Let $s\in[0,2)$, $0<t_{1}<t_{2}\leq T$ and $z_{j}\in G^{s}([t_{1},t_{2}])$ satisfy
\begin{equation*}
|z_{j}^{(i)}(t)|\leq M_{j}\frac{(i!)^{s}}{R^{i}},
\end{equation*}
where $R$ is a positive constant and the positive constants $M_{j}$ are such that 
\begin{equation}\label{19}
\sum\limits_{j=1}\limits^{\infty}M_{j}e^{\sqrt{\lambda_{j}}}<\infty.
\end{equation}
Then the function $u$ defined by \eqref{17} solves system \eqref{A1}-\eqref{A4} and $u\in G^{\frac{s}{2},\frac{s}{2},s}([-1,0]\times[0,1]\times[t_{1},t_{2}])$.
\end{proposition}
\begin{proof} As the proof is similar to that of \cite[Proposition 2.1]{MRRR}, it is only sketched.\par
Let $m,p,q\in \mathbb{N}$. By applying Proposition $\ref{P1}$ and \eqref{17}, we obtain that
\begin{equation*}
\begin{split}
|\partial_{t}^{m}\partial_{x}^{p}\partial_{y}^{q}u(x,y,t)|\leq& C\|\partial_{t}^{m}u(\cdot,\cdot,t)\|_{H^{p+q+2} (\Omega )}\\
\leq& CB^{[\frac{p+q+2}{2}]+1}\sum\limits_{n=0}^{[\frac{p+q+2}{2}]+1}\|P^{n}\partial_{t}^{m}u(\cdot,\cdot,t)\| _{L^2(\Omega )}\\
\leq& CB^{[\frac{p+q+2}{2}]+1}\sum\limits_{n=0}^{[\frac{p+q+2}{2}]+1}\sup\limits_{(x,y)\in \Omega}|\partial_{t}^{m}P^{n}u(x,y,t)|\\
\leq& CB^{[\frac{p+q+2}{2}]+1}\sum\limits_{n=0}^{[\frac{p+q+2}{2}]+1}\sup\limits_{(x,y)\in \Omega}
\sum\limits_{j=1}\limits^{\infty}\sum\limits_{i=0}\limits^{\infty}|\partial_{t}^{m}P^{n}(g_{i,j}(x)z_{j}^{(i)}(t)e_{j}(y))|.
\end{split}
\end{equation*}
By the definitions of $g_{i,j}$ and $e_{j}$, it is clear that
\begin{equation*}
\partial_{t}^{m}P^{n}(g_{i,j}(x)z_{j}^{(i)}(t)e_{i}(y))=
\begin{cases}
    \displaystyle z_{j}^{(i+m)}(t)(-1)^{n}g_{i-n,j}(x)e_{j}(y), \quad& i\geq n; \\
    \displaystyle0, \quad & i<n.
\end{cases}
\end{equation*}
Setting $k=i-n$ and $N=n+m$, arguing as in \cite[Proposition 2.1]{MRRR}, we infer from Corollary \ref{C1} that
\begin{equation*}
\begin{split}
\sum\limits_{j=1}\limits^{\infty}\sum\limits_{i=0}\limits^{\infty}|\partial_{t}^{m}P^{n}(g_{i,j}(x)z_{j}^{(i)}(t)e_{i}(y))|=&\sum\limits_{j=1}\limits^{\infty}\sum\limits_{i=n}\limits^{\infty}|z_{j}^{(i+m)}(t)g_{i-n,j}(x)e_{j}(y)|\\
\leq&C\sum\limits_{j=1}\limits^{\infty}\sum\limits_{k=0}\limits^{\infty}M_{j}\frac{(k+N)!^{s}}{R^{k+N}}e^{\sqrt{\lambda_{j}}}\frac{1}{(2k)!}\\
\leq&C\frac{(N!)^{s}}{(\frac{R}{2^{s}})^{N}}\\
\leq&C\frac{(n!)^{s}(m!)^{s}}{R_{1}^{n}R_{2}^{m}},
\end{split}
\end{equation*}
where $R_{1}=R_{2}=R/4^{s}$.\par
Gathering the above estimates together, we obtain that 
\begin{equation*}
\begin{split}
|\partial_{t}^{m}\partial_{x}^{p}\partial_{y}^{q}u(x,y,t)|
\leq& CB^{[\frac{p+q+2}{2}]+1}\sum\limits_{n=0}^{[\frac{p+q+2}{2}]+1}\sup\limits_{(x,y)\in \Omega}\sum\limits_{j=1}\limits^{\infty}\sum\limits_{i=0}\limits^{\infty}|\partial_{t}^{m}P^{n}(g_{i,j}(x)z_{j}^{(i)}(t)e_{i}(y))|\\
\leq& CB^{[\frac{p+q+2}{2}]+1}\sum\limits_{n=0}^{[\frac{p+q+2}{2}]+1}\frac{(n!)^{s}(m!)^{s}}{R_{1}^{n}R_{2}^{m}}\\
\leq& C\frac{(p!)^{\frac{s}{2}}(q!)^{\frac{s}{2}}(m!)^{s}}{\overline{R}_{1}^{p}\overline{R}_{2}^{q}\overline{R}_{3}^{m}}
\end{split}
\end{equation*}
for some positive constants $\overline{R}_{1},\overline{R}_{2},\overline{R}_{3}$.
Finally, it is easily seen that $u$ is indeed a solution of the ZK system. 
\end{proof}

Let $\overline{u}$ denote the solution of the free evolution for the ZK system:
\begin{eqnarray}\label{10}
\begin{array}{l}
\left\{
\begin{array}{lll}\overline{u}_{t}+a\overline{u}_{x}+\triangle \overline{u}_{x}=0,
\\\overline{u}(-1,y,t)=\overline{u}(0,y,t)=\overline{u}_{x}(0,y,t)=0,
\\\overline{u}(x,0,t)=\overline{u}(x,1,t)=0,\\
\overline{u}(x,y,0)=u_{0}(x,y),
\end{array}
\right.
\end{array}
\begin{array}{lll}(x,y)\in(-1,0)\times(0,1),~t\in(0,T),
\\ y\in(0,1),~t\in(0,T),\\x\in(-1,0),~t\in(0,T),\\x\in(-1,0),~y\in(0,1).
\end{array}
\end{eqnarray}
As for KdV, we have a Kato smoothing effect.  
\begin{proposition}
Let $u_0\in L^2 (\Omega)$. 
\begin{enumerate}
\item  System \eqref{10} admits a unique solution $\overline{u}\in  C([0,T];L^{2}(\Omega)) \cap L^{2}(0,T;H_{0}^{1}(\Omega))$ and we have
\begin{equation}\label{11}
\sup\limits_{t\in[0,T]}\|\overline{u}(\cdot,\cdot,t)\|^{2} _{L^2 (\Omega )}+\int_{0}^{T}\|\overline{u}(\cdot,\cdot,t)\|^{2}_{H^{1} (\Omega) }dt\leq C\|u_{0}\| _{L^2(\Omega) }^{2}.
\end{equation}
\item If, in addition,  $u_{0}\in \mathcal{D}(A)\cap H^{3}(\Omega)$, then $\overline{u}\in 
C([0,T];H^{3}(\Omega))\cap L^{2}(0,T;H^{4}(\Omega))$ and we have
\begin{equation}\label{12}
\sup\limits_{t\in[0,T]}\|\overline{u}(\cdot,\cdot,t)\|^{2}_{H^{3} (\Omega) }+\int_{0}^{T}\|\overline{u}(\cdot,\cdot,t)\|^{2}_{H^{4} (\Omega )}dt\leq C\|u_{0}\|^{2}_{H^{3} (\Omega ) }.
\end{equation}
\end{enumerate}
\end{proposition}
\begin{proof} (i) comes from \cite{STW}. Let us proceed with the proof of (ii). For any $u_{0}\in \mathcal{D}(A)\cap H^{3}(\Omega)$, 
we have that 
$\overline{u}\in C([0,T];\mathcal{D}(A))$ by the semigroup theory, and hence
$\overline{u}\in C([0,T];H^2(\Omega)\cap H_{0}^{1}(\Omega)).$
Let $w_{0}=Au_{0}$ and $w=A\overline{u}$. It is well known that $w$ is the solution of \eqref{10} with initial value $w_{0}\in L^{2}(\Omega)$. According to (i), we have
$$-\triangle \overline{u}_{x}-a\overline{u}_{x}=A\overline{u}=w\in C([0,T];L^{2}(\Omega))\cap L^{2}(0,T;H_{0}^{1}(\Omega)).$$
Therefore $\triangle \overline{u}_{x}\in C([0,T];L^{2}(\Omega))\cap L^{2}(0,T;H^1(\Omega))$. Assume finally that 
$u_0\in {\mathcal D}(A)\cap H^3(\Omega )$, and let us prove that $u\in C([0,T],H^3(\Omega ))\cap L^2(0,T,H^4(\Omega ))$. 
Decompose $u$ as $u(x,y,t)=\sum_{j=1}^\infty \u (x,t)e_j(y)$. Then for $j\ge 1$, $\u$  solves 
\ba
&&\frac{d\u}{dt} + \u ''' +(a-\lj  )\u'=0, \label{BB1}\\
&&\u (-1,t)= \u (0,t)=\u '(0,t)=0, \label{BB2}\\
&&\u (.,0)=\u ^0,\label{BB3}
\ea
where $u_0 (x,y)=\sum_{j=1}^\infty \u ^0(x) e_j(y)$. Multiplying in \eqref{BB1} by $\u$ (resp. by $(x+1)\u$) and integrating over
$(-1,0)_x\times (0,T)_t$, we obtain respectively
\ba
&& \int_{-1}^0 |\u (x,T)|^2 dx + \int_0^T |\u '(-1,t)|^2 dt =\int_{-1}^0 |\u ^0(x)|^2 dx, \label{BB4}\\
&&\int_{-1}^0  (x+1) |\u (x,T)|^2 dx + 3 \int_0^T \int_{-1}^0 |\u ' |^2 dxdt + (\lj -a) \int_0^T  \int_{-1}^0 |\u|^2dxdt   \nonumber\\ 
&&\qquad =\int_{-1}^0 (x+1) |\u ^0(x)|^2 dx. \label{BB5}
\ea
It follows from \eqref{BB4} that for any $k\in \N$
\be
\label{BB25}
\sum_{j=1}^\infty \lj ^k \Vert \u (.,t)\Vert ^2 \le \sum_{j=1}^\infty  \lj ^k \Vert \u ^0\Vert ^2, \qquad \forall t\in \R_+
\ee
(that is, $\Vert \partial _y^k u(.,.,t)\Vert ^2_{L^2(\Omega)}  \le \Vert \partial _y^k u_0\Vert ^2_{L^2(\Omega)}$ for all $t\in \R_+$), and from \eqref{BB5} that 
\be
\label{BB26}
\int_0^T  \sum_{j=1}^\infty  (
\lj ^k \Vert \u '(.,t)\Vert ^2 + \lj ^{k+1} \Vert \u (.,t) \Vert ^2) dt 
 \le (1+aT) \sum_{j=1}^\infty  \lj ^k \Vert \u ^0\Vert ^2, \qquad \forall T>0
\ee
(that is, $\int_0^T \Vert \nabla \partial _y^k u(.,.,t)\Vert_{L^2(\Omega)} ^2dt  \le (1+aT) \Vert \partial _y^k u_0\Vert _{L^2(\Omega)} ^2$ for all $T>0$). We need the following lemma.
\begin{lemma} \label{lem1}
Let $a\ge 0$ and $\lambda >0$ be given. Let $H^k$ ($k\in \N$) denote the Sobolev space $H^k(-1,0)$, and let ${\mathcal H}^3:=\{ u\in H^3(-1,0); \
u(-1)=u(0)=u'(0)=0 \}$. Let $\Vert \cdot \Vert $ denote the norm $\Vert \cdot \Vert _{L^2(-1,0)}$.   \\
1. There exists a constant $C>0$ such that 
\be
\label{BB31}
\sum_{k=0}^3 \lambda ^k\Vert \partial _x ^{3-k} y\Vert ^2  \le C \left( \Vert y''' +(a-\lambda ) y'\Vert ^2 +\lambda ^3 \Vert y\Vert ^2\right)
\qquad \forall y\in {\mathcal H}^3,\ \forall \lambda \ge \lambda _0 . 
\ee
2. There exists a constant $C'>0$ such that 
\ba
\label{BB32}
\sum_{k=0}^4 \lambda ^k\Vert \partial _x ^{4-k} y\Vert ^2  
\le C' \left( \Vert y^{(4)}  +(a-\lambda ) y''\Vert ^2 +
\Vert y'''   + ( a-\lambda ) y'\Vert ^2 +
\lambda ^4 \Vert y\Vert ^2\right)  \nonumber\\
\qquad \forall y\in {\mathcal H}^3\cap H^4,\ \forall \lambda \ge \lambda _0 . 
\ea
\end{lemma}
\noindent
{\em Proof of Lemma \ref{lem1}:} 1. Pick any $y\in {\mathcal H}^3$ and any $\lambda \ge 0$. By the Interpolation Theorem  and Young inequality, we have that 
\begin{eqnarray*}
\lambda ^2 \Vert y'\Vert ^2 \le C\lambda ^2 \Vert y\Vert ^\frac{4}{3} \Vert y'''\Vert ^\frac{2}{3} \le 
\varepsilon \Vert y''' \Vert ^2 + C_\varepsilon \lambda ^3 \Vert y\Vert ^2,\\
\lambda  \Vert y''\Vert ^2 \le C\lambda  \Vert y\Vert ^\frac{2}{3} \Vert y'''\Vert ^\frac{4}{3} \le 
\varepsilon \Vert y''' \Vert ^2 + C'_\varepsilon \lambda ^3 \Vert y\Vert ^2.
\end{eqnarray*} 
We infer  that if $\lambda \ge \lambda _0 > 0$
\begin{eqnarray*}
\Vert y'''\Vert ^2 
&\le&  2\Vert y'''  +(a-\lambda ) y'\Vert ^2 + 2(a-\lambda )^2 \Vert y'\Vert ^2 \\
&\le& 2\Vert y''' +(a-\lambda ) y'\Vert ^2 + 2 \varepsilon \Vert y''' \Vert ^2 
+ 2C_\varepsilon |a-\lambda |^3 \Vert y\Vert ^2\\
&\le& 2\Vert y''' +(a-\lambda ) y'\Vert ^2 + 2 \varepsilon \Vert y''' \Vert ^2 
+ C''_\varepsilon \lambda ^3 \Vert y\Vert ^2
\end{eqnarray*}
and \eqref{BB31} follows by picking $\varepsilon <1/4$. \\
2. Pick now any $y\in {\mathcal H}^3\cap H^4$ and any $\lambda \ge 0$. Then we have 
\begin{eqnarray*}
\lambda ^3 \Vert y'\Vert ^2 \le C\lambda ^3 \Vert y\Vert ^\frac{3}{2} \Vert y\Vert _{H^4} ^\frac{1}{2} \le 
\varepsilon (\Vert y^{(4)} \Vert ^2 +\Vert y'''\Vert ^2) + C_\varepsilon \lambda ^4 \Vert y\Vert ^2,\\
\lambda ^2 \Vert y''\Vert ^2 \le C\lambda ^2 \Vert y \Vert \, \Vert y \Vert _{H^4}\le 
\varepsilon (\Vert y^{(4)} \Vert ^2 +\Vert y'''\Vert ^2) + C'_\varepsilon \lambda ^4 \Vert y\Vert ^2, \\
\lambda \Vert y'''\Vert ^2 \le C\lambda  \Vert y\Vert ^\frac{1}{2} \Vert y\Vert _{H^4}  ^\frac{3}{2} \le 
\varepsilon (\Vert y^{(4)} \Vert ^2 +\Vert y'''\Vert ^2)  + C''_\varepsilon \lambda ^4 \Vert y\Vert ^2.
\end{eqnarray*} 
On the other hand, we have that for $\lambda \ge \lambda _0>0$
\begin{eqnarray*}
\Vert y^{(4)}\Vert ^2 &\le&
 2\Vert y^{(4)} +(a-\lambda ) y''\Vert ^2 + 2(a-\lambda )^2 \Vert y''\Vert ^2\\
&\le& 2\Vert y^{(4)} +(a-\lambda ) y''\Vert ^2 +
2\varepsilon (\Vert y^{(4)} \Vert ^2 +\Vert y'''\Vert ^2) + C'''_\varepsilon \lambda ^4 \Vert y\Vert ^2,
\end{eqnarray*}
and \eqref{BB32}  follows by picking $\varepsilon <1/4$ and by using \eqref{BB31}. 

Assuming that $u_0\in {\mathcal D}(A)\cap H^3(\Omega )$ and using \eqref{BB25} and \eqref{BB31}, 
we obtain that for any $t\in [0,T]$ (with a constant $C$ that may vary from line to line)
\begin{eqnarray*}
\Vert u(.,.,t)\Vert ^2_{H^3(\Omega )} &=&\Vert u(.,.,t)\Vert ^2_{H^2(\Omega)} 
+\sum_{k=0}^3 \Vert \partial _y^k\partial _x^{3-k} u (.,.,t)\Vert ^2_{L^2(\Omega )} \\
&\le& C\Vert u_0\Vert ^2 _{{\mathcal D}(A)} +\sum_{k=0}^3\sum_{j=1}^\infty \lj ^k\Vert \partial _x^{3-k}\u  (.,t)
\Vert ^2\\
&\le& C\Vert u_0\Vert ^2_{{\mathcal D}(A)} + C\sum_{k=0}^3\sum_{j=1}^\infty \big( \Vert \u ''' (.,t) +(a-\lj )\u ' (.,t)\Vert ^2+\lj^3\Vert \u (.,t) \Vert ^2\big)\\
&\le& C\Vert u_0\Vert ^2_{{\mathcal D}(A)} +C\Vert \partial _y^3 u_0\Vert ^2 _{L^2(\Omega )} \\
&\le& C\Vert u_0\Vert ^2_{H^3(\Omega )}. 
\end{eqnarray*}
On the other hand 
$\Vert u(.,.,t)\Vert ^2_{H^4(\Omega )} =\Vert u(.,.,t)\Vert ^2_{H^3(\Omega)} 
+\sum_{k=0}^4 \Vert \partial _y^k\partial _x^{4-k} u (.,.,t)\Vert ^2_{L^2(\Omega )}$ and it is clear that 
$\int_0^T\Vert u(.,.,t)\Vert ^2_{H^3(\Omega )} dt\le C\Vert u_0\Vert ^2_{H^3(\Omega )}$. Using \eqref{BB32}, we obtain
\begin{eqnarray*}
&&\int_0^T \sum_{k=0}^4 \Vert \partial _y^k\partial _x^{4-k} u (.,.,t)\Vert ^2_{L^2(\Omega )} dt\\
&&\qquad = \int_0^T \sum_{k=0}^4\sum_{j=1}^\infty \lj ^k\Vert \partial _x^{4-k}\u  (.,t)\Vert ^2 dt\\
&&\qquad \le C \int_0^T  \sum_{j=1}^\infty \left( \Vert \u ^{(4)} +(a-\lj ) \u '' \Vert ^2 +
\Vert \u '''  +(a-\lj ) \u '\Vert ^2 +
\lj ^4 \Vert \u \Vert ^2\right) dt \\
&&\qquad \le  C\int_0^T \left( \Vert A u(.,.,t)\Vert ^2_{H^1(\Omega )} + \Vert \partial _y^4 u(.,.,t)\Vert ^2_{L^2(\Omega )}\right) dt\\
&&\qquad \le C\Vert u_0\Vert ^2_{H^3(\Omega )}
\end{eqnarray*}
where we used \eqref{BB26} with $k=3$. 
This completes the proof of the proposition.
\end{proof}
Interpolating between \eqref{11} and \eqref{12}, we obtain
\begin{equation*}
\begin{split}
&\sup\limits_{t\in[0,T]}\|\overline{u}(\cdot,\cdot,t)\|^{2}_{H^{1} (\Omega )}+\int_{0}^{T}\|\overline{u}(\cdot,\cdot,t)\|^{2}_{H^{2} (\Omega)}dt\leq C\|u_{0}\|^{2}_{H^{1} (\Omega )},\\
&\sup\limits_{t\in[0,T]}\|\overline{u}(\cdot,\cdot,t)\|^{2}_{H^{2} (\Omega )}+\int_{0}^{T}\|\overline{u}(\cdot,\cdot,t)\|^{2}_{H^{3} (\Omega )}dt\leq C\|u_{0}\|^{2}_{H^{2}
(\Omega )}.
\end{split}
\end{equation*}
This gives
\begin{equation*}
\|\overline{u}(\cdot,\cdot,t)\|_{H^{n+1} (\Omega )}\leq \frac{C}{\sqrt{t}}\|u_{0}\|_{H^{n} (\Omega )},~~~\textrm{for}~n\in\{0,1,2,3\}.
\end{equation*}
\par
Proceeding as in \cite[Proposition 2.2]{MRRR}, we can show that if $u_{0}\in L^{2}(\Omega)$, then $\overline{u}(t)\in \mathcal{D}(A^{n})$ for any $t\in(0,T]$ and $n\in\mathbb{N}$, and it holds
\begin{equation}\label{14}
\|A^{n}\overline{u}(\cdot,\cdot,t)\| _{L^2( \Omega )} \leq \frac{C^{n}}{t^{\frac{3n}{2}}}n^{\frac{3n}{2}}\|u_{0}\| _{L^2(\Omega )} .
\end{equation}
Without loss of generality, we assume that $T=1$. Then for any $p,q\in\mathbb{N}$, we infer from Proposition \ref{P1} that
\begin{equation*}
\begin{split}
|\partial_{x}^{p}\partial_{y}^{q}\overline{u}(x,y,t)|\leq& \|\overline{u}(\cdot,\cdot,t)\|_{H^{p+q+2} (\Omega )}\\
\leq&C_{0}B^{[\frac{p+q+2}{2}]+1}\sum\limits_{n=0}^{[\frac{p+q+2}{2}]+1}\|P^{n}\overline{u}(\cdot,\cdot,t)\| _{L^2(\Omega )} \\
\leq&C_{0}B^{[\frac{p+q+2}{2}]+1}\sum\limits_{n=0}^{[\frac{p+q+2}{2}]+1}\frac{C^{n}}{t^{\frac{3n}{2}}}n^{\frac{3n}{2}}\|u_{0}\| _{L^2(\Omega )}\\
\leq&Ct^{-\frac{3}{2}[\frac{p+q}{2}]-3}\frac{(p!)^{\frac{3}{4}}(q!)^{\frac{3}{4}}}{R_{1}^{p}R_{2}^{q}}
\end{split}
\end{equation*}
 for some $R_{1},R_{2}>0$. This means that $\overline{u}(\cdot,\cdot,t)\in G^{\frac{3}{4},\frac{3}{4}}([-1,0]\times[0,1])$ for any $t\in(0,T]$.
\par
Let
\begin{equation*}
f_{j}(t) :=\int_{0}^{1}e_{j}(y)\partial_{x}^{2}\overline{u}(0,y,t)dy.
\end{equation*}
\begin{lemma}\label{L3}
For any $j\geq1$ and $n\geq0$, there exist positive constants $R_{1},R_{2}$ and $C$ such that
\begin{equation*}
|f^{(n)}_{j}(t)|
\leq \frac{C}{(j\pi)^{j}}t^{-\frac{3}{2}(n+[\frac{j}{2}]+3)}\frac{(n!)^{\frac{3}{2}}(j!)^{\frac{3}{4}}}{R_{1}^{n}R_{2}^{j}} \cdot
\end{equation*}
\end{lemma}
\begin{proof}
Without loss of generality, we can assume that $T=1$. Since $\overline{u}(\cdot,\cdot,t)\in \mathcal{D}(A^{n})$ for any $t\in(0,T]$ and $n\in\mathbb{N}$, it follows from Lemma \ref{lemma0} that $$\partial_{x}^{2}\partial_{y}^{2n}\overline{u}(x,0,t)=\partial_{x}^{2}\partial_{y}^{2n}\overline{u}(x,1,t)=0,~~~\forall x\in[-1,0],\  \forall t\in(0,T], \ \forall n\in\mathbb{N}.$$
Then, integrating by parts $j-$times, we deduce that
\begin{equation}\label{15}
\begin{split}
f_{j}(t)=&\sqrt{2}\int_{0}^{1}\sin(j\pi y)\partial_{x}^{2}\overline{u}(0,y,t)dy\\
=&\frac{\sqrt{2}}{j\pi}\int_{0}^{1}\cos(j\pi y)\partial_{x}^{2}\partial_{y}\overline{u}(0,y,t)dy\\
=&\frac{\sqrt{2}}{(j\pi)^{2}}\int_{0}^{1}\sin(j\pi y)\partial_{x}^{2}\partial_{y}^{2}\overline{u}(0,y,t)dy\\
=&\begin{cases}
\frac{\sqrt{2}}{(j\pi)^{j}}\int_{0}^{1}\sin(j\pi y)\partial_{x}^{2}\partial_{y}^{j}\overline{u}(0,y,t)dy, \quad& \textrm{if}~j~\textrm{is~even} ; \\
-\frac{\sqrt{2}}{(j\pi)^{j}}\int_{0}^{1}\cos(j\pi y)\partial_{x}^{2}\partial_{y}^{j}\overline{u}(0,y,t)dy, \quad & \textrm{if}~j~\textrm{is~odd}.
\end{cases}
\end{split}
\end{equation}
To estimate $|f^{(n)}_{j}(t)|(n\in\mathbb{N})$, it remains to estimate $|\partial_{t}^{n}\partial_{x}^{2}\partial_{y}^{j}\overline{u}(0,y,t)|$.
Let
$$l=[\frac{j+4}{2}]+1.$$
Taking \eqref{8} (with $u=P^{i}\overline{u}$) and \eqref{14} into account, we obtain that
\begin{equation}\label{16}
\begin{split}
|\partial_{t}^{n}\partial_{x}^{2}\partial_{y}^{j}\overline{u}(x,y,t)|
=&|P^{n}\partial_{x}^{2}\partial_{y}^{j}\overline{u}(x,y,t)| \\
\leq &C\|P^{n}\overline{u}(\cdot,\cdot,t)\|_{H^{j+4} (\Omega )}\\
\leq &CB^{l}\sum\limits_{k=0}\limits^{l}\|P^{n+k}\overline{u}(\cdot,\cdot,t)\| _{L^2(\Omega )}   \\
\leq &CB^{l}\sum\limits_{k=0}\limits^{n+l}\|P^{k}\overline{u}(\cdot,\cdot,t)\| _{L^2(\Omega )} \\
\leq &CB^{l}\sum\limits_{k=0}\limits^{n+l}\frac{C^{k}k^{\frac{3}{2}k}}{t^{\frac{3}{2}k}}  \Vert u_0\Vert _{L^2(\Omega )} \\
\leq &CB^{l}\frac{C^{n+l}(n+l+1)(n+l)^{\frac{3}{2}(n+l)}}{t^{\frac{3}{2}(n+l)}}   \Vert u_0\Vert  _{L^2(\Omega )} \\
\leq &Ct^{-\frac{3}{2}(n+[\frac{j}{2}]+3)}\frac{(n!)^{\frac{3}{2}}(j!)^{\frac{3}{4}}}{R_{1}^{n}R_{2}^{j}} \Vert u_0\Vert _{L^2(\Omega )} 
\end{split}
\end{equation}
for some $R_{1},R_{2}>0$.\par
Combining \eqref{15} and \eqref{16}, we obtain
\begin{equation*}
\begin{split}
|f^{(n)}_{j}(t)|\leq &\frac{C}{(j\pi)^{j}}\sup\limits_{y\in[0,1]}|\partial_{t}^{n}\partial_{x}^{2}\partial_{y}^{j}\overline{u}(0,y,t)|\\
\leq &\frac{C}{(j\pi)^{j}}t^{-\frac{3}{2}(n+[\frac{j}{2}]+3)}\frac{(n!)^{\frac{3}{2}}(j!)^{\frac{3}{4}}}{R_{1}^{n}R_{2}^{j}} \cdot
\end{split}
\end{equation*}
\end{proof}
\par
Now, we can prove the first main result in this paper.
\begin{proof}[Proof of Theorem \ref{T1}]
Pick any $\tau\in(0,T)$, $s\in[3/2,2)$ and let
$$z_{j}(t)=\phi_{s}\Big(\frac{t-\tau}{T-\tau}\Big)f_{j}(t),~~~0\leq t\leq T,$$
where
\begin{equation*}
\phi_{s}(\rho)=
\begin{cases}
1 \quad & \textrm{if}~\rho\leq0,\\
0 \quad & \textrm{if}~\rho\geq1,\\
\frac{e^{-\frac{M}{(1-\rho)^{\sigma}}}}{e^{-\frac{M}{\rho^{\sigma}}}+e^{-\frac{M}{(1-\rho)^{\sigma}}}} \quad & \textrm{if}~\rho\in(0,1)
\end{cases}
\end{equation*}
with $M>0$ and $\sigma=(s-1)^{-1}$. As $\phi_{s}$ is Gevrey of order $s$, there exist $R_{\phi}>0$ such that
$$|\phi_{s}^{(p)}(\rho)|\leq C\frac{(p!)^{s}}{R_{\phi}^{p}}~~~\forall~p\in\mathbb{N},\rho\in\mathbb{R}.$$
\par
Then, applying Lemma \ref{L3}, for any $\varepsilon\in(0,T)$ and $t\in[\varepsilon,T]$, we have
\begin{equation*}
\begin{split}
|z^{(i)}_{j}(t)|
\leq &\sum\limits_{n=0}\limits^{i}  \left( \begin{array}{c}  i \\ n \end{array} \right) 
\Big| \partial _t ^{i-n} [ \phi_{s} \Big(\frac{t-\tau}{T-\tau}\Big) ]\Big||f^{(n)}_{j}(t)|\\
\leq &C\sum\limits_{n=0}\limits^{i} \left( \begin{array}{c}  i \\ n \end{array} \right) 
\frac{(i-n)!^{s}}{R_{\phi}^{i-n}}(\frac{1}{T-\tau })^{i-n}
\frac{1}{(j\pi)^{j}}t^{-\frac{3}{2}(n+[\frac{j}{2}]+3)}\frac{(n!)^{\frac{3}{2}}(j!)^{\frac{3}{4}}}{R_{1}^{n}R_{2}^{j}}\\
\leq &C\frac{1}{(j\pi)^{j}}\varepsilon^{-\frac{3}{2}([\frac{j}{2}]+3)}\frac{(j!)^{\frac{3}{4}}}{R_{2}^{j}}
\sum\limits_{n=0}\limits^{i} \left( \begin{array}{c}  i \\ n \end{array} \right) 
 \frac{(i-n)!^{s}}{R_{\phi}^{i-n}}(\frac{1}{T-\tau })^{i-n}\varepsilon^{-\frac{3}{2}n}\frac{(n!)^{\frac{3}{2}}}{R_{1}^{n}}\\
\leq &C\frac{1}{(j\pi)^{j}}\varepsilon^{-\frac{3}{2}([\frac{j}{2}]+3)}\frac{(j!)^{\frac{3}{4}}}{R_{2}^{j}}\frac{(i!)^{s}}{\min\{R_{\phi},R_{1}\}^{i}}\sum\limits_{n=0}\limits^{i} \left( \begin{array}{c}  i \\ n \end{array} \right) 
(\frac{1}{T -\tau })^{i-n}\varepsilon^{-\frac{3}{2}n}\\
\leq & M_{j}\frac{(i!)^{s}}{\widehat{R}^{i}},
\end{split}
\end{equation*}
where $M_{j}$ satisfies \eqref{19}. Let
\begin{equation*}
u(x,y,t)=
\begin{cases}
u_{0}(x,y) \quad & \textrm{if}~x\in[-1,0],y\in[0,1],t=0,\\
\sum\limits_{j=1}\limits^{\infty}\sum\limits_{i=0}\limits^{\infty}g_{i,j}(x)z_{j}^{(i)}(t)e_{j}(y) \quad & \textrm{if}~x\in[-1,0],y\in[0,1],t\in(0,T].
\end{cases}
\end{equation*}
Then, it is easy to see that $u(\cdot,\cdot,T)=0$. By Proposition \ref{P3}, $u\in G^{\frac{s}{2},\frac{s}{2},s}([-1,0]\times[0,1]\times[\varepsilon,T])$ for any $\varepsilon\in(0,T)$.  
Furthermore, we have
\begin{equation*}
\begin{split}
& u_t+ au_x + \Delta u_x=0=\overline{u}_t + a\overline{u}_x + \Delta \overline{u}_x \quad \textrm{ in } 
\Omega \times (0,T), \\
&u(0,y,t)=0=\overline{u}(0,y,t),~~~\forall y\in[0,1],\ \forall t \in(0,\tau),\\
&\partial_{x}u(0,y,t)=0=\partial_{x}\overline{u}(0,y,t),~~~\forall y\in[0,1],\ \forall t\in(0,\tau),\\
&\partial_{x}^{2}u(0,y,t)=\sum\limits_{j=1}\limits^{\infty}z_{j}(t)e_{j}(y)=\partial_{x}^{2}\overline{u}(0,y,t),~~~\forall y\in[0,1],\ \forall  t\in(0,\tau).
\end{split}
\end{equation*}
It follows from Holmgren theorem that $u(x,y,t)=\overline{u}(x,y,t)$ for any $(x,y,t)\in[-1,0]\times[0,1]\times(0,\tau)$. 
In particular,  $u\in C([0,T];L^{2}(\Omega ))$ and $h=0$ for $t\in[0,\tau)$, so that $h\in G^{\frac{s}{2},s}([0,1]\times[0,T])$. The proof of Theorem \ref{T1} is complete.
\end{proof}

\section{Reachable states}
\begin{proposition}\label{P4}
For any $j\geq1$, assume that $z_{j}\in G^{2}([0,T])$ is such that
\begin{equation*}
|z_{j}^{(i)}(t)|\leq M_{j}\frac{(2i)!}{R^{2i}},~~~\forall~i\geq0,~t\in[0,T],
\end{equation*}
where $R>1$ and $M_{j}$ satisfies \eqref{19}. Then the function $u$ defined by \eqref{17} solves system 
\eqref{A1}-\eqref{A4} and $u\in G^{1,1,2}([-1,0]\times[0,1]\times[0,T])$.
\end{proposition}
\begin{proof}
According to the proof of Proposition \ref{P3}, for any $m,p,q\in\mathbb{N}$, we have
\begin{equation*}
\begin{split}
|\partial_{t}^{m}\partial_{x}^{p}\partial_{y}^{q}u(x,y,t)|
\leq CB^{[\frac{p+q+2}{2}]+1}\sum\limits_{n=0}^{[\frac{p+q+2}{2}]+1}\sup\limits_{(x,y)\in \Omega}
\sum\limits_{j=1}\limits^{\infty}\sum\limits_{i=n}\limits^{\infty}|z_{j}^{(i+m)}(t)g_{i-n,j}(x)e_{j}(y)|.
\end{split}
\end{equation*}
Let $k=2i-2n$ and $N=2n+2m$. We can obtain by the same arguments as in \cite[Proposition 3.1]{MRRR} that
\begin{equation*}
\begin{split}
\sum\limits_{j=1}\limits^{\infty}\sum\limits_{i=n}\limits^{\infty}|z_{j}^{(i+m)}(t)g_{i-n,j}(x)e_{j}(y)|
\leq& \sum\limits_{j=1}\limits^{\infty}\sum\limits_{i=n}\limits^{\infty}M_{j}\frac{(2i+2m)!}{R^{2i+2m}}\frac{Ce^{\sqrt{\lambda_{j}}}}{(2i-2n)!}\\
=&\sum\limits_{j=1}\limits^{\infty}CM_{j}e^{\sqrt{\lambda_{j}}}\sum\limits_{k=0}\limits^{\infty}\frac{(k+N)!}{R^{k+N}k!}\\
\leq&C\sum\limits_{k=0}\limits^{\infty}\frac{(k+N)!}{R^{k+N}k!}\\
=&C\sum\limits_{k=0}\limits^{\infty}\frac{(k+1)\cdots(k+N)}{R^{k+N}}\\
\leq&C(\frac{\alpha e}{R^{\sigma}})^{N}N!\sqrt{N}\\
\leq&C\frac{(2n)!(2m)!}{R_{1}^{n}R_{2}^{m}},
\end{split}
\end{equation*}
where $R_{1},R_{2}$ are two positive constants, $\sigma\in(0,1)$ and
$$\alpha=\sup\limits_{k\geq0}\frac{k+2}{(R^{1-\sigma})^{k+1}}.$$\par
It follows from the above estimates that
\begin{equation*}
\begin{split}
|\partial_{t}^{m}\partial_{x}^{p}\partial_{y}^{q}u(x,y,t)|
\leq& CB^{[\frac{p+q+2}{2}]+1}\sum\limits_{n=0}^{[\frac{p+q+2}{2}]+1}\frac{(2n)!(2m)!}{R_{1}^{n}R_{2}^{m}}\\
\leq& C \frac{p!q!(m!)^{2}}{\widehat{R}^{p}_{1}\widehat{R}^{q}_{2}\widehat{R}^{m}_{3}}
\end{split}
\end{equation*}
for some positive constants $\widehat{R}_{1},\widehat{R}_{2}$ and $\widehat{R}_{3}$. This ends the proof of Proposition \ref{P4}.
\end{proof}
\par
As a particular case of \cite[Proposition 3.6]{MRR} (with $a_{0}=1, a_{p}=[2p(2p-1)]^{-1}$ for $p\geq1$), we have the following result.
\begin{proposition}\label{P5}
Let $\{d_{q}\}_{q\geq0}$ be a sequence of real numbers such that
\begin{equation*}
|d_{q}|\leq CH^{q}(2q)!~~\forall~q\geq0
\end{equation*}
for some $H>0$ and $C>0$. Then for all $\widetilde{H}>e^{e^{-1}}H$, there exists a function $f\in C^{\infty}(\mathbb{R})$ such that
\begin{equation*}
\begin{split}
&f^{(q)}(0)=d_{q}~~~\forall~q\geq0,\\
&|f^{(q)}(x)|\leq C\widetilde{H}^{q}(2q)!~~~\forall~q\geq0,~x\in\mathbb{R}.
\end{split}
\end{equation*}
\end{proposition}

\noindent
Let 
\begin{eqnarray*}
{\mathcal X} &:=&\{u\in C^\infty ([-1,0]\times[0,1]); \ \\
&&P^{n}u(0,y)=\partial_{x}P^{n}u(0,y)=P^{n}u(x,0)=P^{n}u(x,1)=0,\ \ \forall n\in 
{\mathbb N} , \ 
\forall x\in [-1,0],\ \forall y\in [0,1] \} .
\end{eqnarray*}
A result similar to Lemma \ref{lemma0} can be derived. 
\begin{lemma}
\label{lemma0ter}\
For any $n\in {\mathbb  N}$,   
we have
\begin{equation}\label{22x}
\partial_{y}^{2n}f(x,0)=\partial_{y}^{2n}f(x,1)=0,~~~\forall f\in {\mathcal X}, \ \forall x\in [-1,0].
\end{equation}
\end{lemma}
\begin{proof}
We proceed by induction on $n$. For $n=0$, \eqref{22x} is obvious since $f\in {\mathcal X}$. 
Assume now that \eqref{22x} is true for $n-1\geq 0$. If $f\in {\mathcal X}$, 
then $Pf \in {\mathcal X}$, so that by the induction hypothesis
\begin{equation*}
\partial_{y}^{2n-2}Pf(x,0)=\partial_{y}^{2n-2}Pf(x,1)=0.
\end{equation*}
This implies
\begin{eqnarray*}
&&\partial_{x}^{3}\partial_{y}^{2n-2}f(x,0)+\partial_{x}\partial_{y}^{2n}f(x,0) +a  \partial _x \partial _y ^{2n-2} f(x,0)=0,\\
&&\partial_{x}^{3}\partial_{y}^{2n-2}f(x,1)+\partial_{x}\partial_{y}^{2n}f(x,1) +a \partial _x \partial _y ^{2n-2} f(x,1)
=0.
\end{eqnarray*}
Since \eqref{22x} is true for $n-1$, we obtain that
\begin{equation*}
\partial_{x}\partial_{y}^{2n}f(x,0)=\partial_{x}\partial_{y}^{2n}f(x,1)=0.
\end{equation*}
This means that for some constants $C_{1}$ and $C_{2}$,
\[
\partial_{y}^{2n}f(x,0)=C_{1},~\partial_{y}^{2n}f(x,1)=C_{2}\quad \forall x\in [-1,0].
\]
On the other hand, we infer from the assumption $f\in {\mathcal X}$ that
\begin{equation*}
\partial_{y}^{2n}f(0,y)=0\quad \forall y\in [0,1].
\end{equation*}
Taking $y=0$ and next $y=1$, we see that $C_{1}=C_{2}=0$. The proof of Lemma \ref{lemma0ter}  is complete.
\end{proof}

\begin{lemma}\label{L2}
If $f\in {\mathcal X}$  is such that
\begin{equation}\label{20}
\int_{0}^{1}e_{l}(y)P^{n}f(0,y)dy=\int_{0}^{1}e_{l}(y)\partial_{x}P^{n}f(0,y)dy=\int_{0}^{1}e_{l}(y)\partial_{x}^{2}P^{n}f(0,y)dy=0
\end{equation}
 for any $l\geq1$ and any $n\geq0$, then
\begin{equation}\label{21}
\int_{0}^{1}e_{l}(y)\partial_{x}^{m}f(0,y)dy=0
\end{equation}
holds  for any $l\geq1$ and any $m\geq0$.
\end{lemma}
\begin{proof}
To prove that \eqref{21} holds for any $l\geq1$ and any $m\geq0$, it is sufficient to show that for any $M\in\mathbb{N}$, \eqref{21} holds for any $l\geq1$ and any $m\leq 3M+2$. We proceed by induction on $M$. 

For $M=0$, we can take $n=0$ in \eqref{20} to see that \eqref{21} holds for any $l\geq1$ and $m\leq 2$.

Assume that \eqref{21} is true for any $l\geq1$ and any $m\leq 3M-1$. We claim that \eqref{21} holds for any $l\geq1$ and $m=3M,3M+1,3M+2$. Indeed, taking $n=M$ in \eqref{20}, we have
\ba
0&=& (-1)^M \int_{0}^{1}e_{l}(y)P^{M}f(0,y)dy\nonumber \\
&=&\int_{0}^{1}e_{l}(y)  (\partial_{x}^{2}+\partial_{y}^{2} +a )^{M} \partial_{x}^{M}f(0,y)dy \nonumber \\
&=&\int_{0}^{1}e_{l}(y)\partial_{x}^{3M}f(0,y)dy \nonumber\\
&&\quad +\int_{0}^{1}e_{l}(y)\sum\limits_{k=0}^{M-1}
\left( \begin{array}{c}  M \\ k \end{array} \right) 
 \sum_{i=0}^{M-k} 
\left( \begin{array}{c}  M-k \\ i \end{array} \right) 
a^{M-k-i} \partial_{x}^{2k+M} \partial_{y}^{2i}f(0,y)dy.
\label{23}
\ea
Since $f\in {\mathcal X}$,
it follows from Lemma \ref{lemma0ter} that 
\begin{equation*}
\partial_{y}^{2n}f(x,0)=\partial_{y}^{2n}f(x,1)=0,~~~\forall x\in[-1,0],\  \forall n\in \mathbb{N}.
\end{equation*}
Then, we obtain by integrations by parts that  for $k\in  \{ 0 , ..., M-1\}$ and $i\in \{ 0, ..., M-k\}$
\[
\int_{0}^{1}e_{l}(y)
\partial_{x}^{2k+M}\partial_{y}^{2i}  f(0,y)dy\\
 =
(-1)^i(l\pi)^{2i}\int_{0}^{1}e_{l}(y)\partial_{x}^{2k+M}f(0,y)dy
=0.
\]
In the last step, we used the fact that  $2k+M\leq 3M-1$. Thus, we infer from \eqref{23} that
\begin{equation*}
\int_{0}^{1}e_{l}(y)\partial_{x}^{3M}f(0,y)dy=0,~~~\forall \, l\geq1.
\end{equation*}
We can show in the same way that \eqref{21} is true for $m=3M+1,3M+2$ by using the fact that
$$\int_{0}^{1}e_{l}(y)\partial_{x}P^{M}f(0,y)dy=\int_{0}^{1}e_{l}(y)\partial^{2}_{x}P^{M}f(0,y)dy=0,~~~\forall \, l\geq1.$$ 
\noindent
The proof of Lemma \ref{L2} is complete.
\end{proof}\par
Now, we are in a position to prove the second main result in this paper.
\begin{proof}[Proof of Theorem \ref{T2}]
Assume that $R:=\min\{R_{1},R_{2}\}>R_{0}=\sqrt[3]{9(a+2)}e^{(3e)^{-1}}$ and pick any $u_{1}\in \mathcal{R}_{R_{1},R_{2}}$. We intend to expand $u_{1}$ in the following form:
\begin{equation*}
u_{1}(x,y)=\sum\limits_{j=1}\limits^{\infty}\sum\limits_{i=0}\limits^{\infty}b_{i,j}g_{i,j}(x)e_{j}(y),
\end{equation*}
where
\begin{equation*}
b_{i,j}= (-1)^i \int_{0}^{1}e_{j}(y)\partial_{x}^{2}P^{i}u_{1}(0,y)dy.
\end{equation*}

Since  $u_{1}\in \mathcal{R}_{R_{1},R_{2}} \subset {\mathcal X}$, 
we have that $P^{i}u_{1}\in {\mathcal X} $ for any $i\in\mathbb{N}$. By Lemma \ref{lemma0ter}, we infer that
\begin{equation*}
\partial_{y}^{2n}P^{i}u_{1}(x,0)=\partial_{y}^{2n}P^{i}u_{1}(x,1)=0,~~~\forall x\in[-1,0].
\end{equation*}
Then, by integration by parts, we have
\begin{equation*}
|b_{i,j}|=|\int_{0}^{1}e_{j}(y)\partial_{x}^{2}P^{i}u_{1}(0,y)dy|\leq \frac{C}{(j\pi)^{j}}\sup\limits_{(x,y)\in \Omega}|\partial_{x}^{2}\partial_{y}^{j}P^{i}u_{1}(x,y)|.
\end{equation*}
Next, we estimate $|\partial_{x}^{2}\partial_{y}^{j}P^{i}u_{1}(x,y)|$.
\begin{equation*}
\begin{split}
|\partial_{x}^{2}\partial_{y}^{j}P^{i}u_{1}(x,y)|
=&|\partial_{x}^{2}\partial_{y}^{j}\sum\limits_{n=0}\limits^{i}
\left( \begin{array}{c}  i \\ n \end{array} \right) 
(\partial_{x}^{2}+\partial_{y}^{2})^{n}\partial_{x}^{n}(a\partial_{x})^{i-n}u_{1}(x,y)|\\
=&|\partial_{x}^{i+2}\partial_{y}^{j}\sum\limits_{n=0}\limits^{i}
\left( \begin{array}{c}  i \\ n \end{array} \right) 
a^{i-n}(\partial_{x}^{2}+\partial_{y}^{2})^{n}u_{1}(x,y)|\\
=&|\partial_{x}^{i+2}\partial_{y}^{j}\sum\limits_{n=0}\limits^{i}
\left( \begin{array}{c}  i \\ n \end{array} \right) 
a^{i-n}\sum\limits_{m=0}\limits^{n}
\left( \begin{array}{c}  n\\ m \end{array} \right) 
\partial_{x}^{2m}\partial_{y}^{2n-2m}u_{1}(x,y)|\\
\leq& \sum\limits_{n=0}\limits^{i}\sum\limits_{m=0}\limits^{n}
\left( \begin{array}{c}  i \\ n \end{array} \right)  \left( \begin{array}{c}  n \\ m \end{array} \right) 
a^{i-n}|\partial_{x}^{2m+i+2}\partial_{y}^{2n-2m+j}u_{1}(x,y)|\\
\leq&C \sum\limits_{n=0}\limits^{i}\sum\limits_{m=0}\limits^{n}
\left( \begin{array}{c}  i \\ n \end{array} \right)  \left( \begin{array}{c}  n \\ m \end{array} \right) 
a^{i-n}\frac{(2m+i+2)!^{\frac{2}{3}}(2n-2m+j)!^{\frac{2}{3}}}{R_{1}^{2m+i+2}R_{2}^{2n-2m+j}}\\
\leq&C \sum\limits_{n=0}\limits^{i}\sum\limits_{m=0}\limits^{n}
\left( \begin{array}{c}  i \\ n \end{array} \right)  \left( \begin{array}{c}  n \\ m \end{array} \right) 
a^{i-n}\frac{(2m+i+2)!^{\frac{2}{3}}(2n-2m+j)!^{\frac{2}{3}}}{R^{2n+i+j+2}}.
\end{split}
\end{equation*}
We notice that
\begin{equation*}
\begin{split}
(2m+i+2)!(2n-2m+j)!=&
\left( \begin{array}{c}  2m+i+2 \\ 2 \end{array} \right) \left( \begin{array}{c}  2n-2m+j \\ j \end{array} \right) 
2!j!(2m+i)!(2n-2m)!\\
\leq &
\left( \begin{array}{c}  2m+i+2 \\ 2 \end{array} \right)  \left( \begin{array}{c}  2n - 2m + j  \\ j \end{array} \right) 
2!j!(2n+i)!,
\end{split}
\end{equation*}
where we used the fact that
$$
\left( \begin{array}{c}  2n+ i \\ 2m+i \end{array} \right) 
=\frac{(2n+i)!}{(2m+i)!(2n-2m)!}\geq1.$$
According to \cite[Lemma A.1]{LR}, we have
$$
\left( \begin{array}{c} 2m+ i +2 \\ 2 \end{array} \right) \left( \begin{array}{c}  2n-2m+j \\ j \end{array} \right) 
\leq 
\left( \begin{array}{c} 2n+  i +j+2 \\ j+ 2 \end{array} \right) 
.$$
This implies
\begin{equation*}
\begin{split}
(2m+i+2)!(2n-2m+j)!
\leq & 
\left( \begin{array}{c}  2n+i+j+2 \\ j+2 \end{array} \right) 
2!j!(2n+i)!\\
=& \frac{(2n+i+j+2)!2!j!(2n+i)!}{(j+2)!(2n+i)!}\\
\leq &(2n+i+j+2)!.
\end{split}
\end{equation*}\par
Combining the above estimates, 
we infer that
\begin{equation*}
\begin{split}
|b_{i,j}|\leq&\frac{C}{(j\pi)^{j}} \sum\limits_{n=0}\limits^{i}\sum\limits_{m=0}\limits^{n}
\left( \begin{array}{c}  i \\ n \end{array} \right)  \left( \begin{array}{c}  n \\ m  \end{array} \right) 
a^{i-n}\frac{(2n+i+j+2)!^{\frac{2}{3}}}{R^{2n+i+j+2}}\\
\leq&\frac{C}{(j\pi)^{j}} \sum\limits_{n=0}\limits^{i} 
\left( \begin{array}{c}  i \\ n \end{array} \right) 
2^{n}a^{i-n}\frac{(3i+j+2)!^{\frac{2}{3}}}{R^{3i+j+2}}\\
=&\frac{C}{(j\pi)^{j}}\frac{(3i+j+2)!^{\frac{2}{3}}}{R^{3i+j+2}}(2+a)^{i}\\
\leq&\frac{C2^{\frac{2}{3}(3i+j+2)}(3i)!^{\frac{2}{3}}(j+2)!^{\frac{2}{3}}(2+a)^{i}}{(j\pi)^{j}R^{3i+j+2}}\\
\leq&\frac{C2^{\frac{2}{3}(j+2)}(j+2)!^{\frac{2}{3}}}{(j\pi)^{j}R^{j+2}}\frac{2^{2i}(3i)!^{\frac{2}{3}}(2+a)^{i}}{R^{3i}}\\
\leq&\frac{C2^{\frac{2}{3}(j+2)}(j+2)!^{\frac{2}{3}}}{(j\pi)^{j}R^{j+2}}\frac{3^{2i}(6\pi i)^{\frac{1}{3}}(4\pi i)^{-\frac{1}{2}}(2i)!(2+a)^{i}}{R^{3i}}\\
\leq&\frac{C2^{\frac{2}{3}(j+2)}(j+2)!^{\frac{2}{3}}}{(j\pi)^{j}R^{j+2}}\frac{[9(2+a)]^{i}(2i)!}{R^{3i}}\\
=&M_{j}\frac{[9(2+a)]^{i}(2i)!}{R^{3i}},
\end{split}
\end{equation*}
where $M_{j}$ satisfies \eqref{19}.\par
By Proposition \ref{P5}, for any $j\geq1$, there exists a function $h_{j}\in G^{2}([0,T])$ and a number 
$\tilde R>1$ such that
\begin{equation}\label{13}
\begin{split}
&h_{j}^{(i)}(T)=b_{i,j}~~~\forall~i\geq0,\\
&|h_{j}^{(i)}(t)|\leq M_{j}\frac{(2i)!}{ \tilde R ^{2i}}~~~\forall~i\geq0,~t\in[0,T].
\end{split}
\end{equation}
Pick any $\tau\in(0,T),s\in(1,2)$ and let
\begin{equation*}
g(t)=1-\phi_{s}\Big(\frac{t-\tau}{T-\tau}\Big)~~~\textrm{for}~t\in[0,T].
\end{equation*}
Setting
\begin{equation*}
z_{j}(t)=h_{j}(t)g(t)~~~\forall~t\in[0,T],
\end{equation*}
following the method developed in \cite[Theorem 3.2]{MRR}, and taking into account the fact that $s<2$, we see  
that $z_{j}$ satisfies
\begin{equation}\label{24}
\begin{split}
&z_{j}^{(i)}(T)=b_{i,j}~~~\forall~j\geq1,~i\geq0,\\
&z_{j}^{(i)}(0)=0~~~\forall~j\geq1,~i\geq0,\\
&|z_{j}^{(i)}(t)|\leq CM_{j}\frac{(2i)!}{\tilde R^{2i}}~~~\forall~j\geq1,~i\geq0,~t\in[0,T],
\end{split}
\end{equation}
where $\tilde R$ is the same as in \eqref{13} and $C$ is a positive constant independent of $i$ and $j$.\par
Let $u$ be as in \eqref{17}. According to \eqref{24}, we have $u_{0}=0$ and
$$u(x,y,T)=\sum\limits_{j=1}\limits^{\infty}\sum\limits_{i=0}\limits^{\infty}g_{i,j}(x)z_{j}^{(i)}(T)e_{j}(y)
=\sum\limits_{j=1}\limits^{\infty}\sum\limits_{i=0}\limits^{\infty}b_{i,j}g_{i,j}(x)e_{j}(y).$$
By  Proposition \ref{P4}, $u$ solves system \eqref{A1}-\eqref{A4} and $u\in G^{1,1,2}([-1,0]\times[0,1]\times[0,T])$. Let
$$h(y,t)=u(-1,y,t)~~~\forall \, y\in[0,1],\forall \, t\in[0,T].$$
Then $h\in G^{1,2}([0,1]\times[0,T])$.\par
Finally, for any $l\geq1$ and $n\geq0$, we have
\begin{equation*}
\begin{split}
&\int_{0}^{1}e_{l}(y)P^{n}u(0,y,T)dy=\int_{0}^{1}e_{l}(y)\sum\limits_{j=1}\limits^{\infty}\sum\limits_{i=n}\limits^{\infty}b_{i,j} (-1)^n g_{i-n,j}(0)e_{j}(y)dy=0\\
&=\int_{0}^{1}e_{l}(y)P^{n}u_{1}(0,y)dy,\\
&\int_{0}^{1}e_{l}(y)\partial_{x}P^{n}u(0,y,T)dy=\int_{0}^{1}e_{l}(y)\sum\limits_{j=1}\limits^{\infty}\sum\limits_{i=n}\limits^{\infty}b_{i,j} (-1)^n g^{\prime}_{i-n,j}(0)e_{j}(y)dy =0\\
&=\int_{0}^{1}e_{l}(y)\partial_{x}P^{n}u_{1}(0,y)dy,\\
&\int_{0}^{1}e_{l}(y)\partial_{x}^{2}P^{n}u(0,y,T)dy=\int_{0}^{1}e_{l}(y)\sum\limits_{j=1}\limits^{\infty}\sum\limits_{i=n}\limits^{\infty}b_{i,j} (-1)^n g^{\prime\prime}_{i-n,j}(0)e_{j}(y)dy =b_{nl}\\
&=\int_{0}^{1}e_{l}(y)\partial_{x}^{2}P^{n}u_{1}(0,y)dy.
\end{split}
\end{equation*}
Since $u(\cdot,\cdot,T),u_{1}\in {\mathcal X}$,  it follows from Lemma \ref{L2} that
\[
\int_{0}^{1}e_{l}(y)[ \partial_{x}^{m} u(0,y,T)  -  \partial _x^mu_1(0,y)]dy=0\quad \forall l\ge 1, \ \forall m\ge 0,
\] 
and hence
\[
\partial_{x}^{m} u(0,y,T)-\partial _x ^m u_1(0,y)=0\quad \forall m\ge 0, \ \forall y\in [0,1]. 
\]
Since the map $x\to u(x,y,T)-u_1(x,y)$ is in $G^1([-1,0])$  (i.e. is analytic) for any $y\in [0,1]$, we infer that
\begin{equation*}
u(x,y,T)=u_{1}(x,y)~~~\forall  (x,y )\in[-1,0]\times[0,1].
\end{equation*}
The proof of Theorem \ref{T2} is complete.
\end{proof}

\section*{Acknowledgements} Lionel Rosier was partially supported by the ANR project Finite4SoS (ANR-15-CE23-0007).
Mo Chen was supported by NSFC Grant (11701078) and China Scholarship Council (No. 201806625055).




\begin{thebibliography}{xx}

\bibitem{adams} R. A. Adams, Sobolev Spaces. Academic Press, New York, 1975.  

\bibitem{cerpa} E. Cerpa, {\em  Control of a Korteweg-de Vries equation: a tutorial}, Math. Control Relat. Fields {\bf 4} (2014), no. 1, 45--99.  






\bibitem{C} M. Chen, {\em Unique continuation property for the Zakharov-Kuznetsov equation},
 Comput. Math. Appl. 77 (2019), no. 5, 1273--1281.    

\bibitem{DL1}  G. G. Doronin, N. A.  Larkin, {\em Stabilization for the linear Zakharov-Kuznetsov equation without critical size restrictions}, J. Math. Anal. Appl. {\bf 428} (2015), no. 1, 337--355.

\bibitem{DL2}  G. G. Doronin, N. A.  Larkin, 
{\em Stabilization of regular solutions for the Zakharov-Kuznetsov equation posed on bounded rectangles and on a strip}, 
Proc. Edinb. Math. Soc. (2) {\bf 58} (2015), no. 3, 661--682.

\bibitem {F1} A. V. Faminski. {\em  The Cauchy problem for the Zakharov?Kuznetsov equation}, 
Differ. Uravn.  {\bf 31} (1995) 1070--1081,
 English transl. in Differential Equ. {\bf 31} (1995). 

\bibitem{F2} A. V. Faminskii,
{\em Initial-boundary value problems in a rectangle for two-dimensional Zakharov-Kuznetsov equation},
J. Math. Anal. Appl. {\bf 463} (2018), no. 2, 760--793.

\bibitem{GG} O. Glass, S. Guerrero, {\em  Some exact controllability results for the linear KdV equation and uniform controllability in the zero-dispersion limit}, Asymptotic Analysis {\bf 60} (2008), 61--100.  

\bibitem{LR} C. Laurent, L. Rosier, {\em Exact Controllability of Semilinear Heat Equations in Spaces of Analytic Functions}, submitted.

\bibitem{LS}   F. Linares and J.-C. Saut, {\em The Cauchy problem for the 3D Zakharov-Kuznetsov equation},
  Discrete Contin. Dynam. Syst. A {\bf 24} (2009), 547--565.
  
 \bibitem{LPS}  F. Linares, A. Pastor and J.-C. Saut, {\em Well-posedness for the ZK equation in a cylinder and on the  background of a KdV Soliton}, Commun. PDEs {\bf 35} (2010), 1674--1689. 

\bibitem{MRR} P. Martin, L. Rosier, P. Rouchon. {\em On the reachable states for the boundary control of the heat equation},  Appl. Math. Res.
Express. AMRX 2016, 2, 181-216.

\bibitem{MRRR} P. Martin, I. Rivas, L. Rosier, P. Rouchon. {\em Exact controllability of a linear Korteweg-de Vries equation by the flatness approach}, SIAM J. Control Optim. {\bf 57} (2019), no. 4, 2467--2486.


\bibitem{PRST} G. Perla-Menzala, L. Rosier, J.-C. Saut, R. Temam, {\em Boundary control of the Zakharov-Kuznetsov equation}, in preparation.  

\bibitem{R1997} L. Rosier, {\em Exact boundary controllability for the Korteweg-de Vries equation on a bounded domain}, ESAIM  Control Optim. Calc. Var. {\bf 2} (1997), 33--55.


\bibitem{R2004} L. Rosier, {\em Control of the surface of a fluid by a wavemaker}, ESAIM Control Optim. Calc. Var. {\bf 10} (3) (2004), 346--380. 

\bibitem{RZsurvey} L. Rosier, B.-Y.  Zhang, {\em Control and stabilization of the Korteweg-de Vries equation: recent progresses.}
J. Syst. Sci. Complex. {\bf 22} (2009), no. 4, 647--682.    

\bibitem{ST} J.-C. Saut, R. Temam, {\em An initial boundary-value problem for the Zakharov-Kuznetsov equation}, Adv. Diff. Equations, {\bf 15} (11-12) (2010), 1001--1031. 

\bibitem{STW} J.-C. Saut, R. Temam and C. Wang. {\em An initial and boundary-value problem for the Zakharov-Kuznestov equation in a bounded domain}, Journal of Mathematical Physics, 2012, 53, 115612.










\end{thebibliography}
\end{document}